\documentclass[oneside,reqno]{amsart}
\usepackage[foot]{amsaddr} 
\usepackage[utf8]{inputenc}
\usepackage{amsmath}
\usepackage{amsfonts}
\usepackage{amssymb}
\usepackage{setspace}
\usepackage{amssymb}
\usepackage[shortlabels]{enumitem}
\usepackage{xcolor}
\usepackage{graphicx}
\usepackage{appendix}
\usepackage{amsthm}
\newtheorem{thm}{Theorem}[section]
\newtheorem{df}{Definition}[section]

\numberwithin{equation}{section}
\usepackage{bbm}
\newtheorem{rmk}{Remark}[section]

\newtheorem{prop}{Proposition}[section]
\newtheorem{lm}{Lemma}[section]

\usepackage{geometry}
\usepackage{fancyhdr}
\usepackage{hyperref}
\hypersetup{
    colorlinks=true,
    linkcolor=blue,
    filecolor=magenta,      
    urlcolor=blue,
    pdftitle={Overleaf Example},
    citecolor=red,
    pdfpagemode=FullScreen,
    }
\usepackage{mfirstuc}
\makeatletter
\def\@setauthors{%
  \begingroup
  \def\thanks{\protect\thanks@warning}%
  \trivlist
  \centering\footnotesize \@topsep30\p@\relax
  \advance\@topsep by -\baselineskip
  \item\relax
  \author@andify\authors
  \def\\{\protect\linebreak}%
  \authors%
  \ifx\@empty\contribs
  \else
    ,\penalty-3 \space \@setcontribs
    \@closetoccontribs
  \fi
  \endtrivlist
  \endgroup
}
\date{}
\begin{document}
\title[Controllability for linear stochastic parabolic equations]{Null controllability for stochastic parabolic equations with Robin boundary conditions}

\author[S. Boulite]{\large Said Boulite$^{1}$ }
\address{$^{1}$ Cadi Ayyad University, National School of Applied Sciences, LMDP, UMMISCO (IRD-UPMC), B.P. 575, Marrakesh, Morocco}
\email{s.boulite@uca.ma}	
\author[A. Elgrou]{\large Abdellatif Elgrou$^{2,\ast}$}
\thanks{\texttt{$^{\ast}$ Corresponding author}}
\address{$^{2}$ Cadi Ayyad University, Faculty of Sciences Semlalia, LMDP, UMMISCO (IRD-UPMC), B.P. 2390, Marrakesh, Morocco}
\email{abdellatif.elgrou@ced.uca.ma}
\author[L. Maniar]{\large Lahcen Maniar$^{2}$}
\email{maniar@uca.ma}

\dedicatory{ \large Dedicated to the memory of Professor Hammadi Bouslous}
\begin{abstract}

We establish the null controllability of forward and backward linear stochastic parabolic equations with linear Robin (or Fourier) boundary conditions. These equations incorporate zero and first order terms with bounded coefficients. To prove our null controllability results, a key tool will be the derivation of two new global Carleman estimates for the weak solutions of the corresponding adjoint equations in negative Sobolev space. These Carleman estimates are established using a duality method.
\end{abstract}
\maketitle
\smallskip
\noindent \textbf{\it Keywords:} {Stochastic parabolic equations, Carleman estimates, Controllability, Observability, Robin boundary conditions.} 

\section{Introduction and main results}
Let $T>0$ and $(\Omega,\mathcal{F},\mathbf{F},\mathbb{P})$ be a fixed complete filtered probability space on which a one-dimensional standard Brownian motion $W(\cdot)$ is defined such that $\mathbf{F}=\{\mathcal{F}_t\}_{t\in[0,T]}$ is the natural filtration generated by $W(\cdot)$ and augmented by all the $\mathbb{P}$-null sets in $\mathcal{F}$. Let $\mathcal{X}$ be a Banach space; we denote by $L^2_{\mathcal{F}_t}(\Omega;\mathcal{X})$ the Banach space of all $\mathcal{X}$-valued $\mathcal{F}_t$-measurable random variables $X$ such that $\mathbb{E}\big(\vert X\vert_\mathcal{X}^2\big)<\infty$, with the canonical norm. $L^2_\mathcal{F}(0,T;\mathcal{X})$ denotes the Banach space consisting of all $\mathcal{X}$-valued $\mathbf{F}$-adapted processes $X(\cdot)$ such that $\mathbb{E}\big(\vert X(\cdot)\vert^2_{L^2(0,T;\mathcal{X})}\big)<\infty$, with the canonical norm. $L^\infty_\mathcal{F}(0,T;\mathcal{X})$ represents the Banach space consisting of all $\mathcal{X}$-valued $\mathbf{F}$-adapted essentially bounded processes, with the canonical norm denoted simply by $|\cdot|_\infty$. $L^2_\mathcal{F}(\Omega;C([0,T];\mathcal{X}))$ indicates the Banach space consisting of all $\mathcal{X}$-valued $\mathbf{F}$-adapted continuous processes $X(\cdot)$ such that $\mathbb{E}\big(\vert X(\cdot)\vert^2_{C([0,T];\mathcal{X})}\big)<\infty$, with the canonical norm and $C([0,T];\mathcal{X})$ denotes the Banach space of all $\mathcal{X}$-valued continuous functions defined on $[0,T]$. Similarly, one can define $L^\infty_\mathcal{F}(\Omega;C^m([0,T];\mathcal{X}))$ for any positive integer $m$.

Let $G\subset\mathbb{R}^N$ ($N\geq1$) be a bounded domain with a $C^2$ boundary $\Gamma$, and $G_0\Subset G$ be a given non-empty open subset strictly contained in $G$ (i.e., $\overline{G_0}\subset G$ where $\overline{G_0}$ denotes the closure of $G_0$). We indicate by $\mathbbm{1}_{G_0}$ the characteristic function of $G_0$ and $dx$ (resp., $d\sigma$) designates the Lebesgue measure (resp., surface measure) in $G$ (resp., on $\Gamma$). Put 
$$Q=(0,T)\times G, \,\,\,\quad \Sigma=(0,T)\times\Gamma\,\,\quad \textnormal{and}\,\,\,\quad Q_0=(0,T)\times G_0.$$ 

Throughout this paper, we assume the following assumptions on the matrix $\mathcal{A}=(a^{jk})_{j,k=1,2,...,N}$:
 \begin{enumerate}[1)]
\item $\mathcal{A}\in L^\infty_\mathcal{F}(\Omega;C^1([0,T];W^{1,\infty}(G;\mathbb{R}^{N\times N})))$ and $a^{jk}=a^{kj}$\,\, for all \,\,$j,k=1,2,...,N$.
\item There exists a constant $c_0>0$ such that
$$\mathcal{A}\xi\cdot\xi\geq c_0\vert\xi\vert^2\quad\textnormal{ for all}\;\;\, (\omega,t,x,\xi)\in\Omega\times Q\times\mathbb{R}^N,$$
where ``\,$\cdot$\,'' is the canonical inner product in $\mathbb{R}^N$ and $|\cdot|$ denotes its associated norm.
\end{enumerate}

The first goal of this paper is to study the null controllability of the following forward stochastic parabolic equation with Robin boundary conditions
\begin{equation}\label{1.1}
\begin{cases}
\begin{array}{ll}
dy - \nabla\cdot(\mathcal{A}\nabla y) \,dt = (a_1 y+B_1\cdot\nabla y +\mathbbm{1}_{G_0}u) \,dt+ (a_2y+B_2\cdot\nabla y+v) \,dW(t) &\textnormal{in}\,\,Q,\\
\partial^\mathcal{A}_\nu y+\beta y=0 &\textnormal{on}\,\,\Sigma,\\
				y(0)=y_0 &\textnormal{in}\,\,G,
			\end{array}
		\end{cases}
\end{equation}
where the coefficients $a_1$, $a_2$, $B_1$, $B_2$ and $\beta$ are assumed to be
$$a_1, a_2\in L_\mathcal{F}^\infty(0, T; L^\infty(G)), \quad B_1, B_2\in L_\mathcal{F}^\infty(0, T;L^\infty(G;\mathbb{R}^N)), \quad \beta\in L_\mathcal{F}^\infty(0, T; L^\infty(\Gamma)).$$
The function $y_0\in L^2(G)$ is the initial state and the control variable consists of the pair 
$$(u,v)\in L^2_\mathcal{F}(0,T;L^2(G_0))\times L^2_\mathcal{F}(0,T;L^2(G)).$$ 
Here, $\partial^\mathcal{A}_\nu y=(\mathcal{A}\nabla y\cdot\nu)|_\Sigma$ denotes the conormal derivative of $y$ where $\nu$ is the outer unit normal vector on $\Gamma$. In what follows, $C$ represents a positive constant depending only on $G$, $G_0$, and $\mathcal{A}$, which may vary from one place to another.

From \cite{yan2018carleman}, one can show that \eqref{1.1} is well-posed i.e., for any initial state
$y_0\in L^2(G)$ and controls $(u,v)\in L^2_\mathcal{F}(0,T;L^2(G_0))\times L^2_\mathcal{F}(0,T;L^2(G))$, the equation \eqref{1.1} has a unique weak solution 
$$y\in L^2_\mathcal{F}(\Omega;C([0,T];L^2(G)))\bigcap L^2_\mathcal{F}(0,T;H^1(G)),$$
that depends continuously on $y_0$, $u$ and $v$.

The first main result of this paper concerns the null controllability of equation \eqref{1.1}.
\begin{thm}\label{thm01.1}
For any initial state $y_0\in L^2(G)$, there exist controls $$(\widehat{u},\widehat{v})\in L^2_\mathcal{F}(0,T;L^2(G_0))\times L^2_\mathcal{F}(0,T;L^2(G)),$$
such that the corresponding solution $\widehat{y}$ of \eqref{1.1} satisfies that $$\widehat{y}(T,\cdot) = 0\,\,\,\; \textnormal{in}\,\,G,\quad\mathbb{P}\textnormal{-a.s.}$$
Moreover, controls $\widehat{u}$ and $\widehat{v}$ can be chosen so that
\begin{align}\label{1.2201}
\begin{aligned}
\vert \widehat{u}\vert^2_{L^2_\mathcal{F}(0,T;L^2(G))}+\vert \widehat{v}\vert^2_{ L^2_\mathcal{F}(0,T;L^2(G))}\leq e^{CK}\,\vert y_0\vert^2_{L^2(G)},
\end{aligned}
\end{align}
where the constant $K$ has the following form
\begin{align*}
K\equiv&\,\,1+\frac{1}{T}+\vert a_1\vert_\infty^{2/3}+\vert a_2\vert_\infty^{2/3}+\vert \beta\vert_\infty^2+\vert B_1\vert_\infty^2+\vert B_2\vert_\infty^2\\
&+T\big(1+\vert a_1\vert_\infty+\vert a_2\vert_\infty^{2}+\vert \beta\vert_\infty^2+\vert B_1\vert_\infty^2\big).
\end{align*}
\end{thm} 

The second goal of this paper is to prove the null controllability of the following backward stochastic parabolic equation with Robin boundary conditions
	\begin{equation}\label{1.2}
		\begin{cases}
			\begin{array}{ll}
			dy+\nabla\cdot(\mathcal{A}\nabla y) \,dt = (a_1y+a_2Y+B\cdot\nabla y+\mathbbm{1}_{G_0}u)\,dt + Y \,dW(t)	 &\textnormal{in}\,\,Q,\\
\partial^\mathcal{A}_\nu y+\beta y=0
&\textnormal{on}\,\,\Sigma,\\
			y(T)=y_T &\textnormal{in}\,\,G,
			\end{array}
		\end{cases}
	\end{equation}
where $y_T\in L^2_{\mathcal{F}_T}(\Omega;L^2(G))$ is the terminal state, $u\in L^2_\mathcal{F}(0,T;L^2(G_0))$ is the control process and the coefficients $a_1$, $a_2$, $B$ and $\beta$ satisfy the following assumptions
$$a_1, a_2\in L^\infty_\mathcal{F}(0,T;L^\infty(G)), \quad B\in L^\infty_\mathcal{F}(0,T;L^\infty(G;\mathbb{R}^N)), \quad \beta\in L_\mathcal{F}^\infty(0, T; L^\infty(\Gamma)).$$

Similarly to \eqref{1.1}, one can also demonstrate that \eqref{1.2} is well-posed i.e., for any $y_T\in L^2_{\mathcal{F}_T}(\Omega;L^2(G))$ and $u\in L^2_\mathcal{F}(0,T;L^2(G_0))$, there exists a unique weak solution $(y,Y)$ of \eqref{1.2} so that
$$(y,Y)\in \Big(L^2_\mathcal{F}(\Omega;C([0,T];L^2(G)))\bigcap L^2_\mathcal{F}(0,T;H^1(G))\Big)\times L^2_\mathcal{F}(0,T;L^2(G)),$$
which depends continuously on the data $y_T$ and $u$.

The second main result of this paper is the following null controllability of equation \eqref{1.2}.
\begin{thm}\label{thm1.2}
For any terminal state $y_T\in L^2_{\mathcal{F}_T}(\Omega;L^2(G))$, there exists a control $\widehat{u}\in L^2_\mathcal{F}(0,T;L^2(G_0))$ such that the corresponding solution $(\widehat{y},\widehat{Y})$ of \eqref{1.2} fulfills that  $$\widehat{y}(0,\cdot) = 0\,\,\,\; \textnormal{in}\,\,G,\quad\mathbb{P}\textnormal{-a.s.}$$ Furthermore, the control $\widehat{u}$ can be chosen in such a way that
\begin{equation}\label{1.44012}
\vert \widehat{u}\vert^2_{L^2_\mathcal{F}(0,T;L^2(G))}\leq e^{CM}\,\vert y_T\vert^2_{L^2_{\mathcal{F}_T}(\Omega;L^2(G))},
\end{equation}
where the constant $M$ is given by
$$M\equiv1+\frac{1}{T}+\vert a_1\vert_\infty^{2/3}+\vert a_2\vert_\infty^2+\vert \beta\vert_\infty^2+\vert B\vert_\infty^2+T\big(1+\vert a_1\vert_\infty+\vert a_2\vert_\infty^2+\vert \beta\vert_\infty^2+\vert B\vert_\infty^2\big).
$$
\end{thm} 
\begin{rmk}
Equation \eqref{1.1} $($with $u\equiv v\equiv 0$$)$ $($respectively, equation \eqref{1.2} with $u\equiv 0$$)$ describe various physical phenomena influenced by stochastic factors and boundary constraints. One such example is the heat equation with stochastic perturbations and Robin boundary conditions, representing the evolution of the temperature distribution $y$ in a medium over time, accounting for random fluctuations and boundary interactions. The term ``$-\partial_\nu^{\mathcal{A}} y$'' at the boundary denotes the inward-directed normal heat flux, possibly scaled by a positive coefficient. Specifically, one may suppose that ``$\beta \geq 0$''; however, this assumption is not required in this paper. Hence, the equation ``$-\partial_\nu^{\mathcal{A}} y = \beta y$'' signifies a linear relationship between this flux and the temperature. For more information regarding the physical model described by stochastic partial differential equations such as \eqref{1.1} and \eqref{1.2}, we refer to \cite[Chapter 5]{lu2021mathematical} and the references cited therein.
\end{rmk}
\begin{rmk}
Inequalities \eqref{1.2201} and \eqref{1.44012} provide an estimate of the cost of the control for systems \eqref{1.1} and \eqref{1.2}, respectively. For instance, for \eqref{1.1}: If Theorem \ref{thm01.1} holds, then we have that $$\mathcal{K}_T \leq \sqrt{e^{CK}},$$
where $\mathcal{K}_T=\mathcal{K}_T(G,G_0,T,a_1,a_2,B_1,B_2)$ defines the control cost for equation \eqref{1.1} at time $T$ i.e.,
$$\mathcal{K}_T=\sup_{|y_0|_{L^2(G)}=1}\,\,\inf_{(u,v)\in \Lambda_T(y_0)} |(u,v)|_{L^2_\mathcal{F}(0,T;L^2(G_0))\times L^2_\mathcal{F}(0,T;L^2(G))},$$
with
$$\Lambda_T(y_0)=\big\{(u,v)\in L^2_\mathcal{F}(0,T;L^2(G_0))\times L^2_\mathcal{F}(0,T;L^2(G)),\quad y(T,\cdot) = 0\,\,\,\; \textnormal{in}\,\,G,\quad\mathbb{P}\textnormal{-a.s.}\big\}.$$
\end{rmk}

In the literature, several results concerning the controllability of deterministic parabolic equations are available. For instance, see \cite{fernandez2006null,fursikov1996controllability,imanuvilov2003carleman,lions1972some}. 
On the other hand, for progress on the controllability of stochastic parabolic equations, we refer the reader to \cite{barbu2003carleman,Preprintelgrou23,SPEwithDBC,liu2014global,lu2011some,fourthorder,lu2021mathematical,tang2009null,yan2018carleman} and the references cited therein. The initial findings regarding the controllability of certain stochastic heat equations are attributed to \cite{barbu2003carleman}. In contrast, the controllability of the general case of both forward and backward stochastic parabolic equations was achieved in \cite{tang2009null} by deriving new global Carleman estimates for backward and forward equations using a weighted identity method. While the previous work addresses the controllability of forward equations with an extra control on the diffusion part, using the spectral method, the author in \cite{lu2011some} establishes the controllability with only one control on the drift when the coefficients are space-independent. In \cite{liu2014global}, an improved Carleman estimate for forward stochastic parabolic equations is found using the duality method. This estimate is utilized to prove the controllability of backward stochastic parabolic equations incorporating zero-order terms with bounded coefficients. For a more extensive analysis of the controllability of forward and backward stochastic parabolic equations, including cases with zero and first order terms with only bounded coefficients, we refer to \cite{Preprintelgrou23}. We remark that all the previous papers deal with the controllability of stochastic parabolic equations with Dirichlet boundary conditions. To the best of our knowledge, \cite{yan2018carleman} is the only paper dealing with Carleman estimates and controllability for stochastic parabolic equations with Robin boundary conditions. The controllability results in \cite{yan2018carleman} are established considering only zero order terms. In the present work, we extend these controllability results to include both zero and first order terms with only bounded coefficients. This extension requires deriving new global Carleman estimates with some source terms in negative Sobolev space.\\

For any parameters $\lambda, \mu\geq1$, we define the following weight functions: For all $(t,x)\in Q$
\begin{equation}\label{2.2012}
\begin{array}{l}
\alpha=\alpha(t,x) = \displaystyle\frac{e^{\mu\psi(x)}-e^{2\mu\vert\psi\vert_\infty}}{t(T-t)},\qquad\;\varphi=\varphi(t,x) = \displaystyle\frac{e^{\mu\psi(x)}}{t(T-t)},\qquad\; \theta=e^{\lambda\alpha},
\end{array}
\end{equation}
where $\psi$ is a function satisfying that
\begin{align}\label{psicond}
\begin{aligned}
\psi\in C^2(\overline{G}),\quad\quad\psi(x)>0 \,\,\;\textnormal{in}\,\,G,\quad\quad&\psi(x)=0\,\,\;\textnormal{on}\,\,\Gamma,\\
\vert\nabla\psi(x)\vert>0\,\;\,\textnormal{in}\,\,\overline{G\setminus G_1},
\end{aligned}
\end{align}
with $G_1\Subset G_0$ is a nonempty open set. The existence of $\psi$ satisfying \eqref{psicond} is proved in \cite{fursikov1996controllability}.
 
We first consider the following general adjoint backward equation 
\begin{equation}\label{202.3}
\begin{cases}
\begin{array}{ll}
dz + \nabla\cdot(\mathcal{A}\nabla z) \,dt = (F_1+\nabla\cdot F) \,dt + Z \,dW(t) &\textnormal{in}\,\,Q,\\
\partial^\mathcal{A}_\nu z-F\cdot\nu=F_2 &\textnormal{on}\,\,\Sigma,\\
z(T)=z_T &\textnormal{in}\,\,G,
			\end{array}
		\end{cases}
\end{equation}
where $z_T\in L^2_{\mathcal{F}_T}(\Omega;L^2(G))$ and $F_1$, $F_2$ and $F$ are assumed to be 
$$F_1\in L^2_\mathcal{F}(0,T;L^2(G)), \quad F_2\in L^2_\mathcal{F}(0,T;L^2(\Gamma)) \quad\textnormal{and} \quad F\in L^2_\mathcal{F}(0,T;L^2(G;\mathbb{R}^N)).$$

To prove Theorem \ref{thm01.1}, we will derive the following Carleman estimate for equation \eqref{202.3}.
\begin{thm}\label{thm02.1}
There exists a large $\mu_0\geq1$ such that for $\mu=\mu_0$, one can find a positive constant $C$ depending only on $G$, $G_0$, $\mu_0$ and $\mathcal{A}$ so that for all $z_T\in L^2_{\mathcal{F}_T}(\Omega;L^2(G))$, $F_1\in L^2_\mathcal{F}(0,T;L^2(G))$, $F_2\in L^2_\mathcal{F}(0,T;L^2(\Gamma))$ and $F\in L^2_\mathcal{F}(0,T;L^2(G;\mathbb{R}^N))$, the weak solution $(z,Z)$ of \eqref{202.3} satisfies that 
\begin{align}\label{1.014}
\begin{aligned}
&\,\lambda^3\mathbb{E}\int_Q \theta^2\varphi^3z^2\,dxdt+\lambda\mathbb{E}\int_Q \theta^2\varphi\vert\nabla z\vert^2\,dxdt+\lambda^2\mathbb{E}\int_\Sigma \theta^2\varphi^2z^2\,d\sigma dt\\
&\leq  C\Bigg[\lambda^3\mathbb{E}\int_{Q_0} \theta^2\varphi^3z^2\,dxdt+\mathbb{E}\int_Q \theta^2 F_1^2\,dxdt+\lambda\mathbb{E}\int_\Sigma \theta^2\varphi F_2^2\,d\sigma dt\\
&\quad\quad\;\,+\lambda^2\mathbb{E}\int_Q \theta^2\varphi^2\vert F\vert^2\,dxdt+\lambda^2\mathbb{E}\int_Q \theta^2\varphi^2Z^2\,dxdt\Bigg],
\end{aligned}
\end{align}
for all $\lambda\geq C(T+T^2)$.
\end{thm}

Secondly, we introduce the following general adjoint forward equation	\begin{equation}\label{3.010}
		\begin{cases}
			\begin{array}{ll}
			dz-\nabla\cdot(\mathcal{A}\nabla z) \,dt = (F_1+\nabla\cdot F)\,dt + F_2 \,dW(t)	 &\textnormal{in}\,\,Q,\\
		\partial^\mathcal{A}_\nu z+F\cdot\nu=F_3 &\textnormal{on}\,\,\Sigma,\\
			 z(0)=z_0 &\textnormal{in}\,\,G,
			\end{array}
		\end{cases}
	\end{equation}
where $z_0\in L^2(G)$ and the processes $F_1$, $F_2$, $F_3$ and $F$ satisfy that
$$F_1,F_2\in L^2_\mathcal{F}(0,T;L^2(G)), \quad F_3\in L^2_\mathcal{F}(0,T;L^2(\Gamma))\quad \textnormal{and} \quad F\in L^2_\mathcal{F}(0,T;L^2(G;\mathbb{R}^N)).$$

The following Carleman estimate for equation \eqref{3.010} will be the key tool to prove Theorem \ref{thm1.2}. 
\begin{thm}\label{thm1.1}
There exists a large $\mu_0\geq1$ such that for $\mu=\mu_0$, one can find a positive constant $C$ depending only on $G$, $G_0$, $\mu_0$ and $\mathcal{A}$ such that for all  $F_1,F_2\in L^2_\mathcal{F}(0,T;L^2(G))$, $F_3\in L^2_\mathcal{F}(0,T;L^2(\Gamma))$, $F\in L^2_\mathcal{F}(0,T;L^2(G;\mathbb{R}^N))$ and $z_0\in L^2(G)$, the weak solution $z$ of \eqref{3.010} fulfills that
\begin{align}
\begin{aligned}
    &\,\lambda^3\mathbb{E}\int_Q \theta^2\varphi^3z^2\,dxdt + \lambda\mathbb{E}\int_Q \theta^2\varphi\vert\nabla z\vert^2\,dxdt+\lambda^2\mathbb{E}\int_\Sigma \theta^2\varphi^2z^2\,d\sigma dt\\
&\leq C\Bigg[\lambda^3\mathbb{E}\int_{Q_0} \theta^2\varphi^3z^2\,dxdt+\mathbb{E}\int_Q \theta^2F_1^2\,dxdt+\lambda\mathbb{E}\int_\Sigma \theta^2\varphi F_3^2\,d\sigma dt\\
\label{3.202002}&\quad\quad\,\;+\lambda^2\mathbb{E}\int_Q \theta^2\varphi^2\vert F\vert^2\,dxdt+\lambda^2\mathbb{E}\int_Q \theta^2\varphi^2F_2^2\,dxdt\Bigg],
\end{aligned}
\end{align}
for all $\lambda\geq C(T+T^2)$.
\end{thm}

In establishing our Carleman estimates \eqref{1.014} and \eqref{3.202002}, we employ the duality method, along with some known Carleman estimates. Specifically, for \eqref{1.014} (resp., \eqref{3.202002}), we will use the known Carleman estimate for backward stochastic parabolic equations (resp., deterministic parabolic equations) with non-homogeneous Neumann boundary conditions as provided in \cite[Theorem 1]{yan2018carleman} (resp., \cite[Theorem 1]{fernandez2006null}). 
\begin{rmk}
It would be quite interesting to study the semilinear case of equations \eqref{1.1} and \eqref{1.2}. For global Lipschitz nonlinearities, it is expected that the controllability of such equations can be achieved following the approach used in the case of Dirichlet boundary conditions, see \cite{san23,semzhxl}. However, for a large class of nonlinearities as discussed in \cite{fernsemilinFour,fernandezzuaua2000}, the problem remains open. This challenge arises from the lack of compactness embeddings for the solution spaces encountered in the analysis of controllability for stochastic partial differential equations.
\end{rmk}
\begin{rmk}
Based on the Carleman estimates \eqref{1.014} and \eqref{3.202002}, one can readily deduce the unique continuation property for solutions of the adjoint equations corresponding to \eqref{1.1} and \eqref{1.2}, respectively. Consequently, employing the duality argument allows us to establish the approximate controllability for both equations \eqref{1.1} and \eqref{1.2}.
\end{rmk}
\begin{rmk}
Studying the null controllability of system \eqref{1.1} with only one control in the drift would be highly interesting. However, this task is still far from completion. For partial progress on this issue, we refer to \cite{lu2011some,observineqback}.
\end{rmk}

The rest of this paper is organized as follows. In Section \ref{sec2w}, we briefly establish the well-posedness of equations \eqref{202.3} and \eqref{3.010} and provide some preliminary tools that will be essential in the sequel. Sections \ref{sec2} and \ref{sec3} are devoted to establishing Theorems \ref{thm02.1} and \ref{thm1.1}, respectively. In Section \ref{sec4}, we apply our Carleman estimates to derive the appropriate observability inequalities for the adjoint systems of \eqref{1.1} and \eqref{1.2}. The proof of Theorems \ref{thm01.1} and \ref{thm1.2} is given in Section \ref{sec5}. Finally, we have included appendices \ref{appsec1} and \ref{appendsecB} in which we prove propositions that will be important in deriving our Carleman estimates.
\section{Well-posedness and some preliminary results}\label{sec2w}
We first recall the definitions of weak solutions of equations \eqref{202.3} and \eqref{3.010}.
\begin{df}
\begin{enumerate}[1.]
\item A process 
$$(z,Z)\in\big(L^2_\mathcal{F}(\Omega;C([0,T];L^2(G)))\bigcap L^2_\mathcal{F}(0,T;H^1(G))\big)\times L^2_\mathcal{F}(0,T;L^2(G))$$ is said to be a weak solution of \eqref{202.3} if for any $t\in[0,T]$ and all $\zeta\in C^1(\overline{G})$, it holds that
\begin{align*}
\int_G(z_T-z(t))\zeta\,dx=&\int_t^T\int_G \mathcal{A}\nabla z\cdot\nabla\zeta \,dxds-\int_t^T\int_\Gamma F_2\zeta \,dxds+\int_t^T\int_G F_1\zeta \,dxds\\
&-\int_t^T\int_G F\cdot\nabla\zeta \,dxds+\int_t^T\int_G Z\zeta \,dxdW(s).
\end{align*}
\item We call a process 
$$z\in L^2_\mathcal{F}(\Omega;C([0,T];L^2(G)))\bigcap L^2_\mathcal{F}(0,T;H^1(G))$$ a weak solution of \eqref{3.010} if for any $t\in[0,T]$ and all $\zeta\in C^1(\overline{G})$, we have that
\begin{align*}
\int_G(z(t)-z_0)\zeta\,dx=&-\int_0^t\int_G \mathcal{A}\nabla z\cdot\nabla\zeta \,dxds-\int_0^t\int_\Gamma F_3\zeta \,dxds+\int_0^t\int_G F_1\zeta \,dxds\\
&-\int_0^t\int_G F\cdot\nabla\zeta \,dxds+\int_0^t\int_G F_2\zeta \,dxdW(s).
\end{align*}
\end{enumerate}
\end{df}
Borrowing some ideas from the approach used in \cite{yan2018carleman}, we establish the following well-posedness of equations \eqref{202.3} and \eqref{3.010}.
\begin{prop}
\begin{enumerate}[1.]
\item Let $z_T\in L^2_{\mathcal{F}_T}(\Omega;L^2(G))$, $F_1\in L^2_\mathcal{F}(0,T;L^2(G))$, $F_2\in L^2_\mathcal{F}(0,T;L^2(\Gamma))$, $F\in L^2_\mathcal{F}(0,T;L^2(G;\mathbb{R}^N))$. Then there exists a unique weak solution
$$(z,Z)\in \Big(L^2_\mathcal{F}(\Omega;C([0,T];L^2(G)))\bigcap L^2_\mathcal{F}(0,T;H^1(G))\Big)\times L^2_\mathcal{F}(0,T;L^2(G))$$
of equation \eqref{202.3} such that
\begin{align*}
&|z|_{L^2_\mathcal{F}(\Omega;C([0,T];L^2(G)))}+|z|_{L^2_\mathcal{F}(0,T;H^1(G))}+|Z|_{L^2_\mathcal{F}(0,T;L^2(G))}\\
&\leq C\, \big(|z_T|_{L^2_{\mathcal{F}_T}(\Omega;L^2(G))}+|F_1|_{L^2_\mathcal{F}(0,T;L^2(G))}+|F_2|_{L^2_\mathcal{F}(0,T;L^2(\Gamma))}\\
&\qquad\;+|F|_{L^2_\mathcal{F}(0,T;L^2(G;\mathbb{R}^N))}\big),
\end{align*}
for some positive constant $C$ depending only on $G$ and $T$.
\item Let $z_0\in L^2(G)$, $F_1\in L^2_\mathcal{F}(0,T;L^2(G))$, $F_2\in L^2_\mathcal{F}(0,T;L^2(G))$, $F_3\in L^2_\mathcal{F}(0,T;L^2(\Gamma))$, and $F\in L^2_\mathcal{F}(0,T;L^2(G;\mathbb{R}^N))$. Then there exists a unique weak solution
$$z\in L^2_\mathcal{F}(\Omega;C([0,T];L^2(G)))\bigcap L^2_\mathcal{F}(0,T;H^1(G))$$
of \eqref{3.010} so that
\begin{align*}
&|z|_{L^2_\mathcal{F}(\Omega;C([0,T];L^2(G)))}+|z|_{L^2_\mathcal{F}(0,T;H^1(G))}\\
&\leq C\,\big(|z_0|_{L^2(G)}+|F_1|_{L^2_\mathcal{F}(0,T;L^2(G))}+|F_2|_{L^2_\mathcal{F}(0,T;L^2(G))}\\
&\qquad\;+|F_3|_{L^2_\mathcal{F}(0,T;L^2(\Gamma))}+|F|_{L^2_\mathcal{F}(0,T;L^2(G;\mathbb{R}^N))}\big),
\end{align*}
for some positive constant $C$ depending only on $G$ and $T$.
\end{enumerate}
\end{prop}

In the subsequent analysis, we frequently utilize weak forms of stochastic processes. Therefore, the following Itô's formula will be crucial for performing certain computations. For further details and proofs, see, for instance, \cite[Chapter 2]{lu2021mathematical}.
\begin{lm}\label{lm1.1}
Let $z,y\in L^2_\mathcal{F}(0,T;H^1(G))$, $z_T\in L^2_{\mathcal{F}_T}(\Omega;L^2(G))$, $y_0\in L^2(G)$, $\phi,\widetilde{\phi}\in L^2_\mathcal{F}(0,T;(H^{1}(G))')$ and $\Psi,Z\in L^2_\mathcal{F}(0,T;L^2(G))$, such that for all $t\in [0,T]$, the processes $(z,Z)$ and $y$ solve respectively the following equations
\begin{align*}
z(t)=z_T-\int_t^T \phi(s)\,ds-\int_t^T Z(s)\,dW(s),\,\,\mathbb{P}\textnormal{-a.s.},\\
y(t)=y_0+\int_0^t \widetilde{\phi}(s)\,ds+\int_0^t \Psi(s)\,dW(s),\,\,\mathbb{P}\textnormal{-a.s.}
\end{align*}
Then the following formulas hold
\begin{enumerate}[1.]
\item For any $t\in[0,T]$, it holds that
\begin{align*}
\vert z(t)\vert^2_{L^2(G)}=&\,\vert z(0)\vert^2_{L^2(G)}+2\int_0^t \langle\phi(s),z(s)\rangle_{(H^{1}(G))',H^1(G)}\,ds\\
&+2\int_0^t (z(s),Z(s))_{L^2(G)}\,dW(s)+\int_0^t \vert Z(s)\vert_{L^2(G)}^2\,ds,\;\mathbb{P}\textnormal{-a.s.}
\end{align*}
\item For all $t\in[0,T]$, we have that
\begin{align*}
(z(t),y(t))_{L^2(G)}=&\,( z(0),y_0)_{L^2(G)}+\int_0^t \langle\phi(s),y(s)\rangle_{(H^{1}(G))',H^1(G)}\,ds\\
&+\int_0^t \langle\widetilde{\phi}(s),z(s)\rangle_{(H^{1}(G))',H^1(G)}\,ds+\int_0^t (y(s),Z(s))_{L^2(G)}\,dW(s)\\
&+\int_0^t (z(s),\Psi(s))_{L^2(G)}\,dW(s)+\int_0^t (Z(s),\Psi(s))_{L^2(G)}\,ds,\;\mathbb{P}\textnormal{-a.s.}
\end{align*}
\end{enumerate}
\end{lm}
In the above lemma, $(\cdot,\cdot)_{L^2(G)}$ denotes the inner product in $L^2(G)$, and $\langle\cdot,\cdot\rangle_{(H^{1}(G))',H^1(G)}$ is the duality product between $H^1(G)$ and its dual $(H^{1}(G))'$ with respect to the pivot space $L^2(G)$.

The divergence operator ``$\nabla\cdot\,$" appearing in equations \eqref{202.3} and \eqref{3.010} needs to be interpreted in the weak sense, as given by the following known result.
\begin{lm}\label{lm1.2}
For any $F\in L^2(G;\mathbb{R}^N)$, we have the following extension of the operator divergence of $F$ as the linear continuous operator on $H^1(G)$ given by
\begin{equation*}
\nabla\cdot F:H^1(G)\longrightarrow\mathbb{R},\;\;\;\; z \longmapsto -\int_G F\cdot\nabla z\, \,dx+\langle F\cdot\nu,z|_\Gamma\rangle_{H^{-1/2}(\Gamma),H^{1/2}(\Gamma)}.
\end{equation*}
\end{lm}

In Section \ref{sec4}, to establish the energy estimate for weak solutions of equations \eqref{202.3} and \eqref{3.010}, we will require the following trace estimate from \cite{fernandez2006null,Grisvard}.
\begin{lm}
 There exists a constant $C=C(G)>0$ so that
\begin{equation}\label{estongamma}
\int_\Gamma z^2 \,d\sigma \leq C\Bigg(\int_G \big(z^2+|\nabla z|^2\big)\,dx\Bigg)^{1/2}\Bigg(\int_G z^2\,dx\Bigg)^{1/2}\quad\textnormal{for all} \;\;z\in H^1(G).
\end{equation}
\end{lm}
\section{Carleman estimate for backward stochastic parabolic equations}\label{sec2}
This section is devoted to proving Theorem \ref{thm02.1}. Firstly, it is easy to check that there exists a positive constant $C=C(G,G_0)$ such that for all $(t,x)\in Q$ 
	\begin{equation}\label{2.301}
		\begin{array}{l}
	\varphi\geq CT^{-2},\qquad\vert\varphi_t\vert\leq CT\varphi^2,\qquad\vert\varphi_{tt}\vert\leq CT^2\varphi^3,\\\\
			\vert\alpha_t\vert\leq CTe^{2\mu\vert\psi\vert_\infty}\varphi^2,\qquad\vert\alpha_{tt}\vert\leq CT^2e^{2\mu\vert\psi\vert_\infty}\varphi^3.
		\end{array}
	\end{equation} 

We will apply the duality method for the proof of Theorem \ref{thm02.1}. Hence, we consider the following controlled forward stochastic parabolic equation 
\begin{equation}\label{2.1}
\begin{cases}
\begin{array}{ll}
dy - \nabla\cdot(\mathcal{A}\nabla y) \,dt = (\lambda^3\theta
^2\varphi^3z+\mathbbm{1}_{G_0}u) \,dt + v \,dW(t) &\textnormal{in}\,\,Q,\\
\partial^\mathcal{A}_\nu y= \lambda^2\theta^2\varphi^2 z&\textnormal{on}\,\,\Sigma,\\
y(0)=0 &\textnormal{in}\,\,G,
			\end{array}
		\end{cases}
\end{equation}
where $(u,v)\in L^2_\mathcal{F}(0,T;L^2(G_0))\times L^2_\mathcal{F}(0,T;L^2(G))$ is the control process, $(z,Z)$ is the solution of \eqref{202.3} and $\theta$ and $\varphi$ are the weight functions provided in \eqref{2.2012}. 

We have the following controllability result for system \eqref{2.1}. For completeness, the proof of this result is given in Appendix \ref{appsec1}.
\begin{prop}\label{prop2.4}
There exists $(\widehat{u},\widehat{v})\in L^2_\mathcal{F}(0,T;L^2(G_0))\times L^2_\mathcal{F}(0,T;L^2(G))$ a pair of controls such that the associated solution $\widehat{y}\in L^2_\mathcal{F}(\Omega;C([0,T];L^2(G)))\cap L^2_\mathcal{F}(0,T;H^1(G))$ of \eqref{2.1} satisfies 
$$\widehat{y}(T,\cdot)=0 \;\;\textnormal{in}\;\, G, \;\,\mathbb{P}\textnormal{-a.s.}$$ 
Moreover, there exists $\mu_0\geq1$ such that for $\mu=\mu_0$, one can find a positive constant $C$ depending only on $G$, $G_0$, $\mu_0$ and $\mathcal{A}$ such that 
\begin{align}
\begin{aligned}
&\,\lambda^{-3}\mathbb{E}\int_{Q} \theta^{-2}\varphi^{-3}\widehat{u}^2\,dxdt+\mathbb{E}\int_Q \theta^{-2}\widehat{y}^2\,dxdt+\lambda^{-1}\mathbb{E}\int_\Sigma\theta^{-2}\varphi^{-1}\widehat{y}^2\,d\sigma dt\\
&+\lambda^{-2}\mathbb{E}\int_Q \theta^{-2}\varphi^{-2}\vert\nabla\widehat{y}\vert^2\,dxdt+\lambda^{-2}\mathbb{E}\int_Q \theta^{-2}\varphi^{-2}\widehat{v}^2\,dxdt\\
\label{2.2}&
\leq C\Bigg[\lambda^3\mathbb{E}\int_Q \theta^2\varphi^3z^2\,dxdt + \lambda^2\mathbb{E}\int_\Sigma \theta^2\varphi^2z^2\,d\sigma dt\Bigg],
\end{aligned}
\end{align}
for all $\lambda\geq C(T+T^2)$.
\end{prop}
Using Proposition \ref{prop2.4}, we are now ready to prove Theorem \ref{thm02.1}. 
\begin{proof}[Proof of Theorem \ref{thm02.1}]
Fix $\mu=\mu_0$ given in Proposition \ref{prop2.4}. Let $(z,Z)$ be the solution of \eqref{202.3} and $\widehat{y}$ be the solution of \eqref{2.1} associated with controls $(\widehat{u},\widehat{v})$ obtained in Proposition \ref{prop2.4}. Applying Itô's formula, computing $d(\widehat{y},z)_{L^2(G)}$, integrating the result on $(0,T)$ and taking the expectation on both sides, we obtain 
\begin{align*}
\begin{aligned}
0=&\,\lambda^3\mathbb{E}\int_Q \theta^2\varphi^3z^2\,dxdt+\lambda^2\mathbb{E}\int_\Sigma \theta^2\varphi^2z^2\,d\sigma dt+\mathbb{E}\int_Q 
 F_1\widehat{y}\,dxdt\\
&
-\mathbb{E}\int_\Sigma F_2\widehat{y}\,d\sigma dt-\mathbb{E}\int_Q F\cdot\nabla\widehat{y}\,dxdt+\mathbb{E}\int_Q \mathbbm{1}_{G_0}\widehat{u}z\,dxdt+\mathbb{E}\int_Q \widehat{v}Z \,dxdt.
\end{aligned}
\end{align*}
Then it follows that
\begin{align}
\begin{aligned}
&\,\lambda^3\mathbb{E}\int_Q \theta^2\varphi^3z^2\,dxdt
+\lambda^2\mathbb{E}\int_\Sigma \theta^2\varphi^2z^2\,d\sigma dt\\
\label{2.27012}&=\mathbb{E}\int_\Sigma F_2 \widehat{y} \,d\sigma dt-\mathbb{E}\int_Q \big[F_1\widehat{y}-F\cdot\nabla\widehat{y}+\mathbbm{1}_{G_0}\widehat{u}z+\widehat{v}Z\big] \,dxdt.
\end{aligned}
\end{align}
Using Young's inequality in the right-hand side of \eqref{2.27012}, we have for any $\varepsilon>0$
\begin{align}
\begin{aligned}
&\,\lambda^3\mathbb{E}\int_Q \theta^2\varphi^3z^2\,dxdt
+\lambda^2\mathbb{E}\int_\Sigma \theta^2\varphi^2z^2\,d\sigma dt\\
&\leq \varepsilon\Bigg[\lambda^{-1}\mathbb{E}\int_\Sigma\theta^{-2}\varphi^{-1}\widehat{y}^2\,d\sigma dt+\mathbb{E}\int_Q\theta^{-2}\widehat{y}^2\,dxdt+\lambda^{-2}\mathbb{E}\int_Q \theta^{-2}\varphi^{-2}\vert\nabla\widehat{y}\vert^2\,dxdt\\
&\quad\quad\;+\lambda^{-3}\mathbb{E}\int_{Q} \theta^{-2}\varphi^{-3}\widehat{u}^2\,dxdt+\lambda^{-2}\mathbb{E}\int_Q \theta^{-2}\varphi^{-2}\widehat{v}^2\,dxdt\Bigg]\\
&\quad +C(\varepsilon)\Bigg[\lambda\mathbb{E}\int_\Sigma \theta^2\varphi F_2^2\,d\sigma dt+\mathbb{E}\int_Q \theta^2 F_1^2\,dxdt+\lambda^2\mathbb{E}\int_Q \theta^2\varphi^2\vert F\vert^2\,dxdt\\
\label{2.3017}&\quad\quad\quad\quad\;\;\,+\lambda^3\mathbb{E}\int_{Q_0} \theta^2\varphi^3z^2\,dxdt+\lambda^2\mathbb{E}\int_Q \theta^2\varphi^2 Z^2\,dxdt\Bigg].
\end{aligned}
\end{align}
Now, recalling the inequality \eqref{2.2}, the estimate \eqref{2.3017} implies that for a sufficiently small $\varepsilon>0$, it holds that
\begin{align}\label{2.014}
\begin{aligned}
&\lambda^3\mathbb{E}\int_Q \theta^2\varphi^3z^2\,dxdt+\lambda^2\mathbb{E}\int_\Sigma \theta^2\varphi^2z^2\,d\sigma dt\\
&\leq\, C\Bigg[\lambda\mathbb{E}\int_\Sigma \theta^2\varphi F_2^2\,d\sigma dt+\mathbb{E}\int_Q \theta^2 F_1^2\,dxdt+\lambda^2\mathbb{E}\int_Q \theta^2\varphi^2\vert F\vert^2\,dxdt\\&\quad\quad\;\;\,+\lambda^3\mathbb{E}\int_{Q_0} \theta^2\varphi^3z^2\,dxdt+\lambda^2\mathbb{E}\int_Q \theta^2\varphi^2 Z^2\,dxdt\Bigg],
\end{aligned}
\end{align}
for $\lambda\geq C(T+T^2)$.

On the other hand, computing $d\big|\theta\varphi^\frac{1}{2}z\big|^2_{L^2(G)}$, integrating the obtained equality on $(0,T)$ and taking the expectation on both sides, we arrive at
\begin{align}\label{2.31012}
\begin{aligned}
2\mathbb{E}\int_Q \theta^2\varphi \mathcal{A}\nabla z\cdot\nabla z \,dxdt=&\,-\mathbb{E}\int_Q (\theta^2\varphi)_tz^2\,dxdt-2\mathbb{E}\int_Q z\mathcal{A}\nabla z\cdot\nabla(\theta^2\varphi)\,dxdt\\
&\,
+2\mathbb{E}\int_\Sigma \theta^2\varphi F_2z \,d\sigma dt-2\mathbb{E}\int_Q \theta^2\varphi F_1z \,dxdt\\
&\,+2\mathbb{E}\int_Q zF\cdot\nabla(\theta^2\varphi)\,dxdt+2\mathbb{E}\int_Q \theta^2\varphi F\cdot\nabla z \,dxdt\\&\,-\mathbb{E}\int_Q \theta^2\varphi Z^2\,dxdt.
\end{aligned}
\end{align}
Recalling \eqref{2.301}, it is easy to check that for a large $\lambda\geq CT^2$, we have
\begin{equation}\label{2.32012}
    \vert (\theta^2\varphi)_t\vert\leq CT\lambda\theta^2\varphi^3,\,\,\;\qquad\;\; \vert\nabla(\theta^2\varphi)\vert\leq C\lambda\theta^2\varphi^2.
\end{equation}
Using \eqref{2.32012} in the right-hand side of \eqref{2.31012}, we deduce that for $\lambda\geq CT^2$
\begin{align}
\begin{aligned}
\mathbb{E}\int_Q \theta^2\varphi\vert\nabla z\vert^2\,dxdt\leq&\, CT\lambda\mathbb{E}\int_Q \theta^2\varphi^3z^2\,dxdt\\
&+C\lambda\mathbb{E}\int_Q \theta^2\varphi^2\vert z\vert\vert\nabla z\vert \,dxdt+C\mathbb{E}\int_\Sigma \theta^2\varphi\vert F_2\vert\vert z\vert \,d\sigma dt\\
\label{2.017}&
    +C\mathbb{E}\int_Q \theta^2\varphi\vert F_1\vert\vert z\vert \,dxdt+C\lambda\mathbb{E}\int_Q \theta^2\varphi^2\vert F\vert\vert z\vert \,dxdt\\
&+C\mathbb{E}\int_Q \theta^2\varphi\vert F\vert\vert\nabla z\vert \,dxdt.
\end{aligned}
\end{align}
By Young's inequality, \eqref{2.017} implies that for all $\varepsilon>0$
\begin{align}
\begin{aligned}
&\,\mathbb{E}\int_Q \theta^2\varphi\vert\nabla z\vert^2\,dxdt\\
&\leq \varepsilon\,\mathbb{E}\int_Q \theta^2\varphi\vert\nabla z\vert^2\,dxdt+C(\varepsilon)\lambda^2\mathbb{E}\int_Q \theta^2\varphi^3z^2\,dxdt+CT\lambda \mathbb{E}\int_Q \theta^2\varphi^3z^2\,dxdt\\
&+C\mathbb{E}\int_\Sigma \theta^2\varphi z^2 \,d\sigma dt+C\mathbb{E}\int_\Sigma \theta^2\varphi F_2^2 \,d\sigma dt+
C\lambda\mathbb{E}\int_Q \theta^2\varphi^2 z^2\,dxdt\\
\label{2.018}&\quad+C\lambda^{-1}\mathbb{E}\int_Q \theta^2 F_1^2 \,dxdt+C\lambda^2\mathbb{E}\int_Q \theta^2\varphi^3 z^2 \,dxdt+C\mathbb{E}\int_Q \theta^2\varphi\vert F\vert^2\,dxdt
\\
&\quad+C(\varepsilon)\mathbb{E}\int_Q \theta^2\varphi\vert F\vert^2 \,dxdt.
\end{aligned}
\end{align}
Taking a small enough $\varepsilon>0$ in \eqref{2.018}, then multiplying the obtained inequality by $\lambda$, we get
\begin{align}\label{2.01836}
\begin{aligned}
\lambda\mathbb{E}\int_Q \theta^2\varphi\vert\nabla z\vert^2\,dxdt\leq&\, C\lambda^3\mathbb{E}\int_Q \theta^2\varphi^3z^2\,dxdt+CT\lambda^2\mathbb{E}\int_Q \theta^2\varphi^3z^2\,dxdt\\
&+C\lambda\mathbb{E}\int_\Sigma \theta^2\varphi z^2 \,d\sigma dt+C\lambda\mathbb{E}\int_\Sigma \theta^2\varphi F_2^2 \,d\sigma dt\\
&+C\lambda^2\mathbb{E}\int_Q \theta^2\varphi^2 z^2\,dxdt+C\mathbb{E}\int_Q \theta^2 F_1^2 \,dxdt\\
&+C\lambda\mathbb{E}\int_Q \theta^2\varphi\vert F\vert^2\,dxdt.
\end{aligned}
\end{align}
We now choose a large $\lambda\geq C(T+T^2)$ in \eqref{2.01836}, we conclude that
\begin{align}\label{2.019}
\begin{aligned}
\lambda\mathbb{E}\int_Q \theta^2\varphi\vert\nabla z\vert^2\,dxdt\leq&\, C\Bigg[\lambda^3\mathbb{E}\int_{Q} \theta^2\varphi^3z^2\,dxdt+\lambda^2\mathbb{E}\int_\Sigma \theta^2\varphi^2z^2\,d\sigma dt\\&\quad\;+\mathbb{E}\int_Q \theta^2 F_1^2\,dxdt
+\lambda\mathbb{E}\int_\Sigma \theta^2\varphi F_2^2\,d\sigma dt\\&\quad\;+\lambda^2\mathbb{E}\int_Q \theta^2\varphi^2\vert F\vert^2\,dxdt\Bigg].
\end{aligned}
\end{align}
Finally, by combining \eqref{2.014} and \eqref{2.019}, we deduce the desired Carleman estimate \eqref{1.014}. This completes the proof of Theorem \ref{thm02.1}.
\end{proof}
\section{Carleman estimate for forward stochastic parabolic equations}\label{sec3} 
This section is addressed to prove Theorem \ref{thm1.1}. Here, we utilize the same weight functions $\theta$ and $\varphi$ defined in \eqref{2.2012}. For any fixed $F_1,F_2\in L_\mathcal{F}^2(0,T;L^2(G))$, $F_3\in L_\mathcal{F}^2(0,T;L^2(\Sigma))$, $F\in L_\mathcal{F}^2(0,T;L^2(G;\mathbb{R}^N))$ and $z_0\in L^2(G)$, let $z$ be the corresponding solution of \eqref{3.010}. By applying the duality method, we first consider $(y,Y)$ the weak solution of the following controlled backward stochastic parabolic equation
\begin{equation}\label{3.6}
\begin{cases}
\begin{array}{ll}
dy + \nabla\cdot(\mathcal{A}\nabla y) \,dt = (\lambda^3\theta^2\varphi^3z + \mathbbm{1}_{G_0}u) \,dt + Y \,dW(t) &\textnormal{in}\,\,Q,\\
\partial^\mathcal{A}_\nu y=-\lambda^2\theta^2\varphi^2 z &\textnormal{on}\,\,\Sigma,\\
y(T)=0 &\textnormal{in}\,\,G,
\end{array}
\end{cases}
\end{equation}
where $u\in L_\mathcal{F}^2(0,T;L^2(G_0))$ is the control process. We establish the following controllability result for equation \eqref{3.6}, where its proof is provided in Appendix \ref{appendsecB}.
\begin{prop}\label{prop3.1}
There exists a control $\widehat{u}\in L^2_\mathcal{F}(0,T;L^2(G_0))$ such that the associated solution $(\widehat{y},\widehat{Y})\in (L^2_\mathcal{F}(\Omega;C([0,T];L^2(G)))\bigcap L^2_\mathcal{F}(0,T;H^1(G)))\times L^2_\mathcal{F}(0,T;L^2(G))$ of \eqref{3.6} satisfies that 
$$\widehat{y}(0,\cdot)=0\;\; \textnormal{in} \;G,\;\;\mathbb{P}\textnormal{-a.s.}$$
Furthermore, there exists $\mu_0\geq1$ such that for $\mu=\mu_0$, one can find a positive constant $C$ depending only on $G$, $G_0$, $\mu_0$ and $\mathcal{A}$ so that for all $\lambda\geq C(T+T^2)$, it holds that
\begin{align}
\begin{aligned}
&\,\lambda^{-3} \mathbb{E}\int_{Q} \theta^{-2}\varphi^{-3}\widehat{u}^2 \,dx dt +\mathbb{E}\int_Q \theta^{-2}\widehat{y}^2 \,dx dt+\lambda^{-1} \mathbb{E}\int_\Sigma \theta^{-2}\varphi^{-1} \widehat{y}^2 \,d\sigma dt\\
&+\lambda^{-2}\mathbb{E}\int_Q \theta^{-2}\varphi^{-2}\vert\nabla\widehat{y}\vert^2\,dx dt+\lambda^{-2}\mathbb{E}\int_Q \theta^{-2}\varphi^{-2}\widehat{Y}^2\,dx dt\\
\label{3.7}&
\leq C\,\Bigg[\lambda^3 \mathbb{E}\int_Q \theta^2\varphi^3 z^2 \,dxdt+\lambda^2 \mathbb{E}\int_\Sigma \theta^2\varphi^2 z^2 \,d\sigma dt\Bigg].
\end{aligned}
\end{align}
\end{prop}

Based on Proposition \ref{prop3.1}, we now prove Theorem \ref{thm1.1}.
\begin{proof}[Proof of Theorem \ref{thm1.1}]
Set $\mu=\mu_0$ provided in Proposition \ref{prop3.1}, and let $(\widehat{y},\widehat{Y})$ be the solution of equation \eqref{3.6} with control $\widehat{u}$ obtained in Proposition \ref{prop3.1}. By Itô's formula, we obtain that
\begin{align*}
\begin{aligned}
    &\,\lambda^3\mathbb{E}\int_Q \theta^2\varphi^3z^2 \,dxdt+\lambda^2\mathbb{E}\int_\Sigma \theta^2\varphi^2z^2 \,d\sigma dt \\
    &=-\mathbb{E}\int_\Sigma F_3 \widehat{y}\,d\sigma dt-\mathbb{E}\int_Q (\widehat{y}F_1-\nabla\widehat{y}\cdot F+\mathbbm{1}_{G_0}\widehat{u}z+\widehat{Y}F_2)\,dxdt.
    \end{aligned}
\end{align*}
Hence, in terms of Young's inequality, it follows that for any $\varepsilon>0$
\begin{align}
\begin{aligned}
&\,\lambda^3\mathbb{E}\int_Q \theta^2\varphi^3z^2 \,dxdt+\lambda^2\mathbb{E}\int_\Sigma \theta^2\varphi^2z^2 \,d\sigma dt\\
&\leq \varepsilon\Bigg[\lambda^{-1}\mathbb{E}\int_\Sigma \theta^{-2}\varphi^{-1}\widehat{y}^2 \,d\sigma dt+\mathbb{E}\int_Q \theta^{-2}\widehat{y}^2 \,dxdt+\lambda^{-2}\mathbb{E}\int_Q \theta^{-2}\varphi^{-2}\vert\nabla\widehat{y}\vert^2 \,dxdt\\
&\qquad\,+\lambda^{-3}\mathbb{E}\int_{Q} \theta^{-2}\varphi^{-3} \widehat{u}^2 \,dxdt+\lambda^{-2}\mathbb{E}\int_Q \theta^{-2}\varphi^{-2} \widehat{Y}^2 \,dxdt\Bigg]\\
\label{2.550}&\quad+C(\varepsilon)\Bigg[\lambda\mathbb{E}\int_\Sigma \theta^2\varphi F_3^2 \,d\sigma dt+\mathbb{E}\int_Q \theta^2F_1^2 \,dxdt+\lambda^2\mathbb{E}\int_Q \theta^2\varphi^2\vert F\vert^2 \,dxdt\\
&\qquad\qquad\;\;+\lambda^3\mathbb{E}\int_{Q_0} \theta^2\varphi^3z^2 \,dxdt+\lambda^2\mathbb{E}\int_Q \theta^2\varphi^2 F_2^2 \,dxdt\Bigg].
\end{aligned}
\end{align}
Using the inequality \eqref{3.7} and choosing a small enough $\varepsilon>0$, the estimate 
\eqref{2.550} provides that for any $\lambda\geq C(T+T^2)$
\begin{align}
\begin{aligned}\label{ine4.8}
&\,\lambda^3\mathbb{E}\int_Q \theta^2\varphi^3z^2 \,dxdt+\lambda^2\mathbb{E}\int_\Sigma \theta^2\varphi^2z^2 \,d\sigma dt\\
&\leq\, C\Bigg[\lambda^3\mathbb{E}\int_{Q_0} \theta^2\varphi^3z^2 \,dxdt+\mathbb{E}\int_Q \theta^2F_1^2 \,dxdt+\lambda\mathbb{E}\int_\Sigma \theta^2\varphi F_3^2 \,d\sigma dt\\
&\qquad\;\;+\lambda^2\mathbb{E}\int_Q \theta^2\varphi^2\vert F\vert^2 \,dxdt+\lambda^2\mathbb{E}\int_Q \theta^2\varphi^2 F_2^2 \,dxdt\Bigg].
\end{aligned}
\end{align}

On the other hand, computing $d\big|\theta\varphi^\frac{1}{2}z\big|^2_{L^2(G)}$, integrating the equality on $(0,T)$ and taking the expectation on both sides, we conclude that
\begin{align}
\begin{aligned}
2\mathbb{E}\int_Q \theta^2\varphi \mathcal{A}\nabla z\cdot\nabla z\,d xdt=&\,\mathbb{E}\int_Q (\theta^2\varphi)_tz^2\,dxdt-2\mathbb{E}\int_Q z\mathcal{A}\nabla z\cdot\nabla(\theta^2\varphi) \,dxdt\\
\label{2.6}& -2\mathbb{E}\int_Q zF\cdot\nabla(\theta^2\varphi)\,dxdt-2\mathbb{E}\int_Q \theta^2\varphi F\cdot\nabla z \,dxdt\\
&+2\mathbb{E}\int_Q \theta^2\varphi zF_1\,dxdt+2\mathbb{E}\int_\Sigma \theta^2\varphi zF_3\,d\sigma dt+\mathbb{E}\int_Q \theta^2\varphi F_2^2\,dxdt.
\end{aligned}
\end{align}
Observe that for a large $\lambda\geq CT^2$, we have
\begin{equation}\label{2.7}
    \vert (\theta^2\varphi)_t\vert\leq CT\lambda\theta^2\varphi^3,\;\qquad\quad \vert\nabla(\theta^2\varphi)\vert\leq C\lambda\theta^2\varphi^2.
\end{equation}
Combining \eqref{2.6} and \eqref{2.7}, we deduce that
\begin{align}
\begin{aligned}
&\,\mathbb{E}\int_Q \theta^2\varphi\vert\nabla z\vert^2\,dxdt\\
&\leq CT\lambda \mathbb{E}\int_Q \theta^2\varphi^3z^2\,dxdt+C\lambda\mathbb{E}\int_Q \theta^2\varphi^2\vert z\vert\vert\nabla z\vert \,dxdt\\
\label{2.8}&
\quad+C\lambda\mathbb{E}\int_Q \theta^2\varphi^2\vert F\vert\vert z\vert \,dxdt+C\mathbb{E}\int_Q \theta^2\varphi\vert F\vert\vert\nabla z\vert \,dxdt\\
& \quad+C\mathbb{E}\int_Q \theta^2\varphi\vert F_1\vert\vert z\vert \,dxdt+C\mathbb{E}\int_\Sigma \theta^2\varphi\vert F_3\vert\vert z\vert \,d\sigma dt+C\mathbb{E}\int_Q \theta^2\varphi F_2^2\,dxdt.
\end{aligned}
\end{align}
By applying Young's inequality on the right-hand side of \eqref{2.8}, we absorb the gradient terms using the left-hand side, then we arrive at 
\begin{align}
\begin{aligned}
\mathbb{E}\int_Q \theta^2\varphi\vert\nabla z\vert^2\,dxdt\leq&\, C\Bigg[T\lambda \mathbb{E}\int_Q \theta^2\varphi^3z^2\,dxdt+\lambda^2\mathbb{E}\int_Q \theta^2\varphi^3z^2\,dxdt\\
\label{2.9}&\quad\;
+\mathbb{E}\int_Q \theta^2\varphi\vert F\vert^2\,dxdt+\lambda^{-1}\mathbb{E}\int_Q \theta^2 F_1^2\,dxdt\\
&\quad\;+\lambda\mathbb{E}\int_Q \theta^2\varphi^2z^2\,dxdt+\lambda\mathbb{E}\int_\Sigma \theta^2\varphi^2z^2\,d\sigma dt\\
&\quad\;+\lambda^{-1}\mathbb{E}\int_\Sigma \theta^2 F_3^2\,d\sigma dt+\mathbb{E}\int_Q \theta^2\varphi F_2^2\,dxdt\Bigg].
\end{aligned}
\end{align}
Finally, multiplying \eqref{2.9} by $\lambda$, then combining the resulting inequality with \eqref{ine4.8} and choosing a sufficiently large $\lambda\geq C(T+T^2)$, we deduce the desired Carleman estimate \eqref{3.202002}. This concludes the proof of Theorem \ref{thm1.1}.
\end{proof}
\section{Observability inequalities for adjoint equations}\label{sec4}
This section is devoted to establishing the observability inequalities for the adjoint equations of \eqref{1.1} and \eqref{1.2}. Through the classical duality argument, such inequalities will play a crucial role in proving our controllability results for equations \eqref{1.1} and \eqref{1.2}.

Firstly, we consider the following adjoint equation of \eqref{1.1}
\begin{equation}\label{1.3}
\begin{cases}
\begin{array}{ll}
dz + \nabla\cdot(\mathcal{A}\nabla z) \,dt = \left[-a_1 z-a_2Z+\nabla\cdot(zB_1+ZB_2)\right] \,dt + Z \,dW(t) &\textnormal{in}\,\,Q,\\
\partial^\mathcal{A}_\nu z+\beta z-zB_1\cdot\nu-ZB_2\cdot\nu=0 &\textnormal{on}\,\,\Sigma,\\
z(T)=z_T &\textnormal{in}\,\,G,
			\end{array}
		\end{cases}
\end{equation}
where $z_T\in L^2_{\mathcal{F}_T}(\Omega;L^2(G))$. From Section \ref{sec2w}, we deduce that \eqref{1.3} is well-posed i.e., for any $z_T\in L^2_{\mathcal{F}_T}(\Omega;L^2(G))$, there exists a unique weak solution $$(z,Z)\in \Big(L^2_\mathcal{F}(\Omega;C([0,T];L^2(G)))\bigcap L^2_\mathcal{F}(0,T;H^1(G))\Big)\times L^2_\mathcal{F}(0,T;L^2(G)),$$
that depends continuously on $z_T$.  We now prove the following observability inequality of \eqref{1.3}. 
\begin{thm}\label{thm4.1}
For any $z_T\in L^2_{\mathcal{F}_T}(\Omega;L^2(G))$, the corresponding solution $(z,Z)$ of equation \eqref{1.3} satisfies that
\begin{equation}\label{1.04}
\vert z(0)\vert^2_{L^2(G)} \leq e^{CK}\,\Bigg[\mathbb{E}\int_{Q_0} z^2 \,dxdt+\mathbb{E}\int_{Q} Z^2 \,dxdt\Bigg],
\end{equation}    
where $K$ and $C$ are the same constants as in Theorem \ref{thm01.1}.
\end{thm}
\begin{proof}
Using Carleman estimate \eqref{1.014} with $$F_1=-a_1z-a_2Z, \,\quad F_2=-\beta z\,\quad \textnormal{and} \,\quad F=zB_1+ZB_2,$$
we deduce that for $\lambda\geq C(T+T^2)$, it holds that
\begin{align}
\begin{aligned}
&\,\lambda^3\mathbb{E}\int_Q \theta^2\varphi^3z^2\,dxdt+\lambda^2\mathbb{E}\int_\Sigma \theta^2\varphi^2z^2\,d\sigma dt\\
&\leq  C\lambda^3\mathbb{E}\int_{Q_0} \theta^2\varphi^3z^2\,dxdt+C\mathbb{E}\int_Q \theta^2 \vert a_1z+a_2Z\vert^2\,dxdt+C\lambda\mathbb{E}\int_\Sigma \theta^2\varphi |\beta z|^2\,d\sigma dt\\
\label{1.0142}&\quad +C\lambda^2\mathbb{E}\int_Q \theta^2\varphi^2\vert zB_1+ZB_2\vert^2\,dxdt+C\lambda^2\mathbb{E}\int_Q \theta^2\varphi^2Z^2\,dxdt.
\end{aligned}
\end{align}
It is now sufficient to choose $\lambda\geq CT^2(\vert a_1\vert_\infty^{2/3}+\vert a_2\vert_\infty^{2/3})$ to obtain
\begin{equation}\label{4.20120}
    C\mathbb{E}\int_Q \theta^2 \vert a_1z+a_2Z\vert^2\,dxdt \leq \frac{1}{4}\lambda^3\mathbb{E}\int_Q \theta^2\varphi^3z^2\,dxdt +  C\lambda^{3}\mathbb{E}\int_Q \theta^2\varphi^3Z^2\,dxdt.
\end{equation}
Also, by $\lambda\geq CT^2|\beta|_\infty^2$, one can find that
\begin{equation}\label{5.33012}
C\lambda\mathbb{E}\int_\Sigma \theta^2\varphi |\beta z|^2\,d\sigma dt\leq \frac{1}{2}\lambda^2\mathbb{E}\int_\Sigma \theta^2\varphi^2 z^2\,d\sigma dt,
\end{equation}
and by $\lambda\geq CT^2(\vert B_1\vert_\infty^2+\vert B_2\vert_\infty^2)$, we have
\begin{align}\label{4.50120}
\begin{aligned}
&\,C\lambda^2\mathbb{E}\int_Q \theta^2\varphi^2\vert zB_1+ZB_2\vert^2\,dxdt \\
    &\leq \frac{1}{4}\lambda^3\mathbb{E}\int_Q \theta^2\varphi^3z^2\,dxdt +  C\lambda^{3}\mathbb{E}\int_Q \theta^2\varphi^3Z^2\,dxdt.
\end{aligned}
\end{align}
Combining \eqref{1.0142}, \eqref{4.20120}, \eqref{5.33012} and \eqref{4.50120}, we conclude that
\begin{equation*}
    \mathbb{E}\int_Q \theta^2 \varphi^3z^2 \,dxdt \leq C\mathbb{E}\int_{Q_0} \theta^2\varphi^3z^2\,dxdt +  C\mathbb{E}\int_Q \theta^2\varphi^3Z^2\,dxdt.
\end{equation*}
Taking  
\begin{align}\label{lamt1}
\lambda\geq\widetilde{\lambda}=C\Big[T+T^2\big(1+\vert a_1\vert_\infty^{2/3}+\vert a_2\vert_\infty^{2/3}+\vert \beta\vert_\infty^2+\vert B_1\vert_\infty^2+\vert B_2\vert_\infty^2\big)\Big],
\end{align} 
we get that
\begin{equation}\label{2.3701}
\mathbb{E}\int_{T/4}^{3T/4}\int_G \theta^2\varphi^3z^2\,dxdt\leq  C\Bigg[\mathbb{E}\int_{Q_0} \theta^2\varphi^3z^2\,dxdt+\mathbb{E}\int_Q \theta^2\varphi^3Z^2\,dxdt\Bigg].
\end{equation}
We now fix $\lambda=\widetilde{\lambda}$ given in \eqref{lamt1} and see that the function
$$t\longmapsto\min_{x\in\overline{G}}\,\big[\theta^2(t,x)\varphi^3(t,x)\big]$$
reaches its minimum in $[T/4,3T/4]$ at $t=T/4$, and
$$t\longmapsto\max_{x\in\overline{G}}\,\big[\theta^2(t,x)\varphi^3(t,x)\big]$$
reaches its maximum in $(0,T)$ at $t=T/2$. Therefore from \eqref{2.3701}, it is easy to deduce that
\begin{equation*}
     \mathbb{E}\int_{T/4}^{3T/4}\int_G z^2 \,dxdt \leq Ce^{C\widetilde{\lambda}T^{-2}} \,\Bigg[ \mathbb{E}\int_{Q_0} z^2 \,dxdt+\mathbb{E}\int_{Q} Z^2 \,dxdt\Bigg].
 \end{equation*}
This implies that
\begin{equation}\label{2.41012}
     \mathbb{E}\int_{T/4}^{3T/4}\int_G z^2 \,dxdt \leq e^{C r_1} \,\Bigg[ \mathbb{E}\int_{Q_0} z^2 \,dxdt+\mathbb{E}\int_{Q} Z^2 \,dxdt\Bigg],
 \end{equation}
 where $r_1=1+\frac{1}{T}+\vert a_1\vert_\infty^{2/3}+\vert a_2\vert_\infty^{2/3}+\vert \beta\vert_\infty^2+\vert B_1\vert_\infty^2+\vert B_2\vert_\infty^2$.
 
On the other hand, let $t\in(0,T)$, computing $d|z|^2_{L^2(G)}$, integrating the equality on $(0,t)$ and taking the mean value, we obtain that
\begin{align}\label{2.42012}
\begin{aligned}
&\,\mathbb{E}\int_G \vert z(0)\vert^2\,dx-\mathbb{E}\int_G \vert z(t)\vert^2\,dx+2c_0\mathbb{E}\int_0^t\int_G \vert\nabla z\vert^2\,dxds\\
&\leq 2|\beta|_\infty\mathbb{E}\int_0^t\int_\Gamma z^2\,d\sigma ds+2\vert a_1\vert_\infty\mathbb{E}\int_0^t\int_G z^2\,dxds+2\vert a_2\vert_\infty\mathbb{E}\int_0^t\int_G \vert z\vert\vert Z\vert \,dxds\\
&\quad\,-\mathbb{E}\int_0^t\int_G Z^2 \,dxds+2\vert B_1\vert_\infty\mathbb{E}\int_0^t\int_G \vert z\vert\vert\nabla z\vert \,dxds+2\vert B_2\vert_\infty\mathbb{E}\int_0^t\int_G \vert Z\vert\vert\nabla z\vert \,dxds.
\end{aligned}
\end{align}
Using the trace estimate \eqref{estongamma} for the first term on the right-hand side of \eqref{2.42012} and applying Young's inequality, we arrive at
\begin{align*}
\begin{aligned}
\mathbb{E}\int_G \vert z(0)\vert^2\,dx &\leq \mathbb{E}\int_G \vert z(t)\vert^2\,dx+Cr_2\mathbb{E}\int_0^t\int_G z^2\,dxds+C\vert B_2\vert_\infty^2\mathbb{E}\int_Q Z^2 \,dxdt,
\end{aligned}
\end{align*}
where $r_2=1+\vert a_1\vert_\infty+\vert a_2\vert^2_\infty+\vert \beta\vert_\infty^2+\vert B_1\vert^2_\infty$. Hence, in terms of the Gronwall inequality, it follows that
\begin{equation}\label{2.45012}
\mathbb{E}\int_G \vert z(0)\vert^2\,dx \leq e^{Cr_2T}\mathbb{E}\int_G \vert z(t)\vert^2\,dx+C\vert B_2\vert_\infty^2\mathbb{E}\int_Q Z^2 \,dxdt.
\end{equation}
Integrating \eqref{2.45012} on $(T/4,3T/4)$ and combining the obtained inequality with \eqref{2.41012}, we end up with 
\begin{equation}\label{4.130012}
    \mathbb{E}\int_G \vert z(0)\vert^2\,dx \leq e^{C\big(1+\frac{1}{T}+r_1+r_2T+\vert B_2\vert_\infty^2\big)}\Bigg[\mathbb{E}\int_{Q_0} z^2\,dxdt+\mathbb{E}\int_Q Z^2\,dxdt\Bigg].
\end{equation}
Finally, from \eqref{4.130012}, it is straightforward to deduce the desired observability inequality \eqref{1.04}.
\end{proof}

Let us now introduce the following adjoint equation of \eqref{1.2}
	\begin{equation}\label{1.5}
		\begin{cases}
			\begin{array}{ll}
			dz-\nabla\cdot(\mathcal{A}\nabla z) \,dt = \left[-a_1z+\nabla\cdot(zB)\right]\,dt - a_2z \,dW(t)	 &\textnormal{in}\,\,Q,\\
	\partial^\mathcal{A}_\nu z+\beta z+zB\cdot\nu=0
   &\textnormal{on}\,\,\Sigma,\\
			z(0)=z_0 &\textnormal{in}\,\,G,
			\end{array}
		\end{cases}
	\end{equation}
where $z_0\in L^2(G)$. From Section \ref{sec2w}, we have the following well-posedness of \eqref{1.5}: There exists a unique weak solution $z$ of \eqref{1.5} so that
$$z\in L^2_\mathcal{F}(\Omega;C([0,T];L^2(G)))\bigcap L^2_\mathcal{F}(0,T;H^1(G)),$$
and that depends continuously on $z_0$. 
The observability inequality of \eqref{1.5} is stated as follows.
\begin{thm}\label{thm4.2}
For any $z_0\in L^2(G)$, the associated solution $z$ of \eqref{1.5} satisfies that 
\begin{equation}\label{1.8}
    \vert z(T)\vert^2_{L^2_{\mathcal{F}_T}(\Omega;L^2(G))} \leq e^{CM}\,\mathbb{E}\int_{Q_0} z^2 \,dxdt,
\end{equation}
where the constants $C$ and $M$ are similar to those defined in Theorem \ref{thm1.2}.
\end{thm}
\begin{proof}
By applying Carleman estimate \eqref{3.202002} with
$$F_1=-a_1z, \quad F_2=-a_2z, \quad F_3=-\beta z\quad  \textnormal{and}\quad F=zB,$$
we deduce that for all $\lambda\geq C(T+T^2)$
\begin{align}
\begin{aligned}
&\lambda^3\mathbb{E}\int_Q \theta^2\varphi^3z^2\,dxdt + \lambda^2\mathbb{E}\int_\Sigma \theta^2\varphi^2z^2\,d\sigma dt\\
&\leq C\Bigg[\lambda^3\mathbb{E}\int_{Q_0} \theta^2\varphi^3z^2\,dxdt+\mathbb{E}\int_Q \theta^2\vert a_1z\vert^2\,dxdt+\lambda\mathbb{E}\int_\Sigma \theta^2\varphi\vert\beta z\vert^2\,d\sigma dt\\
\label{3.13}& \quad\quad\;\,+\lambda^2\mathbb{E}\int_Q \theta^2\varphi^2\vert a_2z\vert^2\,dxdt+\lambda^2\mathbb{E}\int_Q \theta^2\varphi^2\vert zB\vert^2\,dxdt\Bigg].
\end{aligned}
\end{align}
It is easy to see that for $\lambda\geq CT^2\big(\vert a_1\vert^{2/3}_\infty+\vert a_2\vert^{2}_\infty+\vert B\vert^{2}_\infty\big)$, it holds that
\begin{align}
\begin{aligned}
&\,C\mathbb{E}\int_Q \theta^2\vert a_1z\vert^2\,dxdt+C\lambda^2\mathbb{E}\int_Q \theta^2\varphi^2\vert a_2z\vert^2\,dxdt
\\
\label{3.14}&+C\lambda^2\mathbb{E}\int_Q \theta^2\varphi^2\vert zB\vert^2\,dxdt\leq\frac{1}{2}\lambda^3\mathbb{E}\int_Q \theta^2\varphi^3z^2\,dxdt,
\end{aligned}
\end{align}
and for $\lambda\geq CT^2|\beta|^2_\infty$, we have that
\begin{equation}\label{secesongam}
C\lambda\mathbb{E}\int_\Sigma \theta^2\varphi\vert\beta z\vert^2\,d\sigma dt\leq \frac{1}{2}\lambda^2\mathbb{E}\int_\Sigma \theta^2\varphi^2 z^2 \,d\sigma dt.
\end{equation}
Combining \eqref{3.13}, \eqref{3.14} and \eqref{secesongam}, we obtain that for \begin{align}\label{lamt2}
\lambda\geq\widetilde{\lambda}= C\big[T+T^2\big(1+\vert a_1\vert^{2/3}_\infty+\vert a_2\vert^{2}_\infty+\vert \beta\vert^{2}_\infty+\vert B\vert^{2}_\infty\big)\big],
\end{align}
the following estimate holds
\begin{equation}\label{3.15}
\mathbb{E}\int_Q \theta^2\varphi^3z^2\,dxdt 
 \leq C\,\mathbb{E}\int_{Q_0} \theta^2\varphi^3z^2\,dxdt.
 \end{equation}
Fixing $\lambda=\widetilde{\lambda}$ defined in \eqref{lamt2}, and  using the same optimization argument as in the proof of Theorem \ref{thm4.1}, the inequality \eqref{3.15} implies that
 \begin{equation}\label{3.18}
\mathbb{E}\int_{T/4}^{3T/4}\int_G z^2 \,dxdt \leq e^{C r_1} \, \mathbb{E}\int_{Q_0} z^2 \,dxdt,
 \end{equation}
where $r_1=1+\frac{1}{T}+\vert a_1\vert_\infty^{2/3}+\vert a_2\vert_\infty^{2}+\vert \beta\vert^{2}_\infty+\vert B\vert_\infty^2$.

For any $t\in(0,T)$, calculating $d|z|^2_{L^2(G)}$, then integrating the result on $(t,T)$, and taking the expectation, we obtain that
\begin{align*}
\begin{aligned}
&\,\mathbb{E}\int_G  z^2(T,x) \,dx-\mathbb{E}\int_G  z^2(t,x) \,dx\\
&= -2\mathbb{E}\int_t^T\int_G \mathcal{A}\nabla z\cdot\nabla z \,dxds-2\mathbb{E}\int_t^T\int_\Gamma \beta z^2\,d\sigma ds-2\mathbb{E}\int_t^T\int_G a_1z^2\,dxds\\
&\,\quad-2\mathbb{E}\int_t^T\int_G zB\cdot\nabla z \,dxds+\mathbb{E}\int_t^T\int_G  a_2^2 z^2\,dxds.
\end{aligned}
\end{align*}
Applying Young's inequality and the trace estimate \eqref{estongamma}, it follows that
\begin{equation}\label{4.24032}
\mathbb{E}\int_G  z^2(T,x) \,dx\leq\mathbb{E}\int_G  z^2(t,x) \,dx+C\big(1+|\beta|_\infty^2+\vert a_1\vert_\infty+\vert B\vert_\infty^2+\vert a_2\vert^2_\infty\big)\mathbb{E}\int_t^T\int_G z^2\,dxds.
\end{equation}
Hence, by Gronwall's inequality, the estimate \eqref{4.24032} provides that 
\begin{equation}\label{3.20}
\mathbb{E}\int_G  z^2(T,x) \,dx\leq e^{CTr_2}\mathbb{E}\int_G  z^2(t,x) \,dx,
\end{equation}
where $r_2=1+|\beta|_\infty^2+\vert a_1\vert_\infty+\vert B\vert_\infty^2+\vert a_2\vert^2_\infty$. Finally, combining \eqref{3.18} and \eqref{3.20}, we conclude the desired observability inequality  \eqref{1.8}.
\end{proof}

\section{Null controllability results}\label{sec5}
This section is devoted to proving the null controllability results for equations \eqref{1.1} and \eqref{1.2} i.e., Theorems \ref{thm01.1} and \ref{thm1.2}. Here, we only show Theorem \ref{thm01.1}, the proof of Theorem \ref{thm1.2} can be provided similarly.
\begin{proof}[Proof of Theorem \ref{thm01.1}]
Let $\varepsilon>0$ and $y_0\in L^2(G)$. Consider the following minimization problem
\begin{align}\label{05.1}
\textnormal{inf} \big\{J_\varepsilon(u,v)\,\big\vert\, (u,v)\in L^2_\mathcal{F}(0,T;L^2(G))\times L^2_\mathcal{F}(0,T;L^2(G))\big\},
\end{align}
with
$$J_\varepsilon(u,v)=\frac{1}{2}\mathbb{E}\int_Q u^2 \,dxdt+\frac{1}{2}\mathbb{E}\int_Q v^2 \,dxdt+\frac{1}{2\varepsilon}\mathbb{E}\int_G \vert y(T)\vert^2\,dx,$$
where $y$ is the solution of equation \eqref{1.1} with initial condition $y_0$ and controls $u$ and $v$. It is easy to see that the functional $J_\varepsilon$ is well-defined, continuous, strictly convex, and coercive. Then the problem \eqref{05.1} admits a unique solution $(u_\varepsilon,v_\varepsilon)\in L^2_\mathcal{F}(0,T;L^2(G))\times L^2_\mathcal{F}(0,T;L^2(G))$.

By the standard duality argument (see, e.g., \cite{imanuvilov2003carleman,lions1972some}), the solution $(u_\varepsilon,v_\varepsilon)$ can be characterized by 
\begin{align}\label{05.2}
    u_\varepsilon=-\mathbbm{1}_{G_0}z_\varepsilon\qquad\textnormal{and}\qquad v_\varepsilon=-Z_\varepsilon,
\end{align}
where $(z_\varepsilon,Z_\varepsilon)$ is the solution of the following backward stochastic parabolic equation
\begin{equation*}
\begin{cases}
\begin{array}{ll}
dz_\varepsilon + \nabla\cdot(\mathcal{A}\nabla z_\varepsilon) \,dt = \left[-a_1 z_\varepsilon-a_2Z_\varepsilon+\nabla\cdot(z_\varepsilon B_1+Z_\varepsilon B_2)\right] \,dt + Z_\varepsilon \,dW(t) &\textnormal{in}\,\,Q,\\
\partial_\nu^\mathcal{A}z_\varepsilon+\beta z_\varepsilon-z_\varepsilon B_1\cdot\nu-Z_\varepsilon B_2\cdot\nu=0 &\textnormal{on}\,\,\Sigma,\\
z_\varepsilon(T)=\frac{1}{\varepsilon}y_\varepsilon(T) &\textnormal{in}\,\,G,
			\end{array}
		\end{cases}
\end{equation*}
with $y_\varepsilon$ is the solution of \eqref{1.1} associated to controls $u_\varepsilon$ and $v_\varepsilon$. By Itô's formula, see that
\begin{align}\label{05.3}
d(y_\varepsilon, z_\varepsilon)_{L^2(G)}=(\mathbbm{1}_{G_0}u_\varepsilon, z_\varepsilon)_{L^2(G)}+(v_\varepsilon, Z_\varepsilon)_{L^2(G)}.
\end{align}
From \eqref{05.3} and \eqref{05.2}, we find that
\begin{align*}
\mathbb{E}\int_{Q_0}z_\varepsilon^2\,dxdt+\mathbb{E}\int_{Q}Z_\varepsilon^2\,dxdt+\frac{1}{\varepsilon}\mathbb{E}\int_G \vert y_\varepsilon(T)\vert^2\,dx=\mathbb{E}\int_G y_0 z_\varepsilon(0)\,dx.
\end{align*}
By Cauchy–Schwarz and Young's inequalities, it follows that
\begin{align}\label{05.5}
\begin{aligned}
&\,\mathbb{E}\int_{Q_0}z_\varepsilon^2\,dxdt+\mathbb{E}\int_{Q}Z_\varepsilon^2\,dxdt+\frac{1}{\varepsilon}\mathbb{E}\int_G \vert y_\varepsilon(T)\vert^2\,dx\\
&\leq \frac{e^{CK}}{2}\vert y_0\vert^2_{L^2(G)}+\frac{1}{2e^{CK}}\vert z_\varepsilon(0)\vert^2_{L^2(G)},
\end{aligned}
\end{align}
where $e^{CK}$ is the same constant appearing in \eqref{1.04}.\\
Combining \eqref{05.5} and \eqref{1.04} and then using \eqref{05.2}, we obtain that
\begin{align}\label{05.6}
\vert u_\varepsilon\vert^2_{L^2_\mathcal{F}(0,T;L^2(G))}+\vert v_\varepsilon\vert^2_{L^2_\mathcal{F}(0,T;L^2(G))}+\frac{2}{\varepsilon}\mathbb{E}\int_G \vert y_\varepsilon(T)\vert^2\,dx\leq e^{CK}\vert y_0\vert^2_{L^2(G)}.
\end{align}
From \eqref{05.6}, we deduce that there exist two processes
$$(\widehat{u},\widehat{v})\in L^2_\mathcal{F}(0,T;L^2(G_0))\times L^2_\mathcal{F}(0,T;L^2(G)),$$ 
and sub-sequences $(u_\varepsilon)$ and $(v_\varepsilon)$ of respectively $(u_\varepsilon)$ and $(v_\varepsilon)$ such that as $\varepsilon\rightarrow0$, we have
\begin{align}\label{05.7}
\begin{aligned}
&\,u_\varepsilon \longrightarrow \widehat{u},\,\,\,\textnormal{weakly in}\,\,\,L^2((0,T)\times\Omega;L^2(G));
\\
&v_\varepsilon \longrightarrow \widehat{v},\,\,\,\textnormal{weakly in}\,\,\,L^2((0,T)\times\Omega;L^2(G)).
\end{aligned}
\end{align}
Finally, if $\widehat{y}$ denotes the solution of \eqref{1.1} associated to the initial state $y_0$ and controls $\widehat{u}$ and $\widehat{v}$, then according to \eqref{05.7} and \eqref{05.6}, we conclude that
$$\vert \widehat{u}\vert^2_{L^2_\mathcal{F}(0,T;L^2(G))}+\vert \widehat{v}\vert^2_{ L^2_\mathcal{F}(0,T;L^2(G))}\leq e^{CK}\,\vert y_0\vert^2_{L^2(G)},$$
and $\widehat{y}(T,\cdot)=0$ in $G$, $\mathbb{P}\textnormal{-a.s.}$ This completes the proof of Theorem \ref{thm01.1}.
\end{proof}

\appendix
\section{proof of Proposition \ref{prop2.4}}\label{appsec1}

We first recall the following known Carleman estimate for equation \eqref{202.3} (with $F\equiv0$ and non-homogeneous Neumann boundary conditions). For the proof, we refer to \cite[Theorem 1]{yan2018carleman}.
\begin{lm}\label{lm2.2}
There exists constants $C>0$ and $\mu_0\geq1$ depending only on $G$, $G_0$ and $\mathcal{A}$ so that for any $z_T\in L^2_{\mathcal{F}_T}(\Omega;L^2(G))$, $F_1\in L^2_\mathcal{F}(0,T;L^2(G))$, $F_2\in L^2_\mathcal{F}(0,T;L^2(\Gamma))$, the weak solution $(z,Z)$ of \eqref{202.3} $($with $F\equiv0$$)$ satisfies that
\begin{align}\label{3.6carles}
\begin{aligned}
&\,\lambda^3\mu^4\mathbb{E}\int_Q \theta^2\varphi^3z^2\,dxdt+\lambda\mu^2\mathbb{E}\int_Q \theta^2\varphi\vert\nabla z\vert^2\,dxdt+\lambda^2\mu^3\mathbb{E}\int_\Sigma \theta^2\varphi^2z^2\,d\sigma dt\\
&\leq  C\Bigg[\lambda^3\mu^4\mathbb{E}\int_{Q_0} \theta^2\varphi^3z^2\,dxdt+\mathbb{E}\int_Q \theta^2 F_1^2\,dxdt+\lambda\mu\mathbb{E}\int_\Sigma \theta^2\varphi F_2^2\,d\sigma dt\\
&\quad\quad\;\,+\lambda^2\mu^2\mathbb{E}\int_Q \theta^2\varphi^2Z^2\,dxdt\Bigg],
\end{aligned}
\end{align}
for all $\mu\geq\mu_0$ and $\lambda\geq C(e^{2\mu\vert\psi\vert_\infty}T+T^2)$.
\end{lm}

Let us now give the Proof of Proposition \ref{prop2.4}. 
\begin{proof}[Proof of Proposition \ref{prop2.4}]
Here, the proof is based on the classical duality method known as the penalized Hilbert Uniqueness Method, which was developed in \cite{GlowiLions}. Let $\mu=\mu_0$ given in Lemma \ref{lm2.2} and $\varepsilon>0$. We define \begin{align}\label{alptheeps}
\alpha_\varepsilon\equiv\alpha_\varepsilon(t,x)=\frac{e^{\mu_0\psi(x)}-e^{2\mu_0\vert\psi\vert_\infty}}{(t+\varepsilon)(T-t+\varepsilon)},\quad\quad\theta_\varepsilon=e^{\lambda\alpha_\varepsilon}.
\end{align}
Let us consider the following minimization problem
\begin{equation}\label{2.03}
    \inf\big\{ J_\varepsilon(u,v)\,\vert\, (u,v)\in\mathcal{H}\big\},
\end{equation}
where 
\begin{align*}
 J_\varepsilon(u,v)=&\,\frac{1}{2}\mathbb{E}\int_{Q_0} \lambda^{-3}\theta^{-2}\varphi^{-3}u^2\,dxdt+\frac{1}{2}\mathbb{E}\int_Q \lambda^{-2}\theta^{-2}\varphi^{-2}v^2\,dxdt\\
 &+\frac{1}{2}\mathbb{E}\int_Q \theta^{-2}_\varepsilon y^2\,dxdt+\frac{1}{2}\mathbb{E}\int_\Sigma \lambda^{-1}\theta_\varepsilon^{-2}\varphi^{-1}y^2\,d\sigma dt\\
 &+\frac{1}{2\varepsilon}\mathbb{E}\int_G |y(T)|^2\,dx,
 \end{align*}
 and
\begin{align*}
\mathcal{H}=\Bigg\{(u&,v)\in L^2_\mathcal{F}(0,T;L^2(G_0))\times L^2_\mathcal{F}(0,T;L^2(G)):\\
&\mathbb{E}\int_{Q_0} \theta^{-2}\varphi^{-3}u^2\,dxdt<\infty\,,\quad\mathbb{E}\int_Q \theta^{-2}\varphi^{-2}v^2\,dxdt<\infty\Bigg\},
\end{align*}
with $y$ is the solution of \eqref{2.1} with controls $u$ and $v$.

It is easy to see that $J_\varepsilon$ is well-defined, continuous, strictly convex, and coercive. Therefore, \eqref{2.03} admits a unique optimal solution $(u_\varepsilon,v_\varepsilon)\in\mathcal{H}$. Moreover, by the standard duality argument (see, for instance, \cite{imanuvilov2003carleman,lions1972some}), this optimal solution can be characterized as follows:
\begin{equation}\label{2.003}
 u_\varepsilon=-\mathbbm{1}_{G_0}\lambda^3\theta^2\varphi^3r_\varepsilon\qquad\textnormal{and}\qquad v_\varepsilon=-\lambda^2\theta^2\varphi^2R_\varepsilon,
 \end{equation}
where $(r_\varepsilon,R_\varepsilon)$ is the solution of the following backward stochastic parabolic equation
\begin{equation}\label{2.04}
\begin{cases}
\begin{array}{ll}
dr_\varepsilon + \nabla\cdot(\mathcal{A}\nabla r_\varepsilon) \,dt = -\theta^{-2}_\varepsilon y_\varepsilon \,dt + R_\varepsilon \,dW(t) &\textnormal{in}\,\,Q,\\
\partial^\mathcal{A}_\nu r_\varepsilon=\lambda^{-1}\theta_\varepsilon^{-2}\varphi^{-1}y_\varepsilon&\textnormal{on}\,\,\Sigma,\\
r_\varepsilon(T)=\frac{1}{\varepsilon}y_\varepsilon(T) &\textnormal{in}\,\,G,
			\end{array}
		\end{cases}
\end{equation}
and $y_\varepsilon$ is the solution of \eqref{2.1} with controls $u_\varepsilon$ and $v_\varepsilon$. Computing $d(y_\varepsilon r_\varepsilon)$ by Itô's formula and using \eqref{2.003}, we get that
\begin{align}
\begin{aligned}
&\,\frac{1}{\varepsilon}\mathbb{E}\int_G \vert y_\varepsilon(T)\vert^2\,dx+\mathbb{E}\int_{Q}\theta_\varepsilon^{-2}y_\varepsilon^2\,dxdt+\lambda^{-1}\mathbb{E}\int_\Sigma\theta_\varepsilon^{-2}\varphi^{-1}y_\varepsilon^2\,d\sigma dt\\
&+\lambda^3\mathbb{E}\int_{Q_0}\theta^2\varphi^3r_\varepsilon^2\,dxdt+\lambda^2\mathbb{E}\int_{Q}\theta^2\varphi^2R_\varepsilon^2\,dxdt\\
\label{2.005}&
= \lambda^3\mathbb{E}\int_{Q}\theta^2\varphi^3 zr_\varepsilon \,dxdt+\lambda^2\mathbb{E}\int_{\Sigma}\theta^2\varphi^2z r_\varepsilon \,d\sigma dt.
\end{aligned}
\end{align}
Applying Young's inequality in the right-hand side of \eqref{2.005}, we have that for any $\rho>0$
\begin{align}
\begin{aligned}
&\,\frac{1}{\varepsilon}\mathbb{E}\int_G \vert y_\varepsilon(T)\vert^2\,dx+\mathbb{E}\int_{Q}\theta_\varepsilon^{-2}y_\varepsilon^2\,dxdt+\lambda^{-1}\mathbb{E}\int_\Sigma\theta_\varepsilon^{-2}\varphi^{-1}y_\varepsilon^2\,d\sigma dt\\
&+\lambda^3\mathbb{E}\int_{Q_0}\theta^2\varphi^3r_\varepsilon^2\,dxdt+\lambda^2\mathbb{E}\int_{Q}\theta^2\varphi^2R_\varepsilon^2\,dxdt\\&
\leq \rho\Bigg[\lambda^3\mathbb{E}\int_{Q}\theta^2\varphi^3r_\varepsilon^2\,dxdt+\lambda^2\mathbb{E}\int_{\Sigma}\theta^2\varphi^2r_\varepsilon^2\,d\sigma dt\Bigg]\\
\label{2.00151402}&\quad+C(\rho)\Bigg[\lambda^3\mathbb{E}\int_{Q}\theta^2\varphi^3z^2\,dxdt+\lambda^2\mathbb{E}\int_{\Sigma}\theta^2\varphi^2z^2\,d\sigma dt\Bigg].
\end{aligned}
\end{align}
Using Carleman estimate \eqref{3.6carles} for equation \eqref{2.04} and the fact that $\theta^2\theta_\varepsilon^{-2}\leq1$, we can deduce that for $\lambda\geq C(T+T^2)$, the following inequality holds
\begin{align}
\begin{aligned}
&\,\frac{1}{\varepsilon}\mathbb{E}\int_G \vert y_\varepsilon(T)\vert^2\,dx+\mathbb{E}\int_{Q}\theta_\varepsilon^{-2}y_\varepsilon^2\,dxdt+\lambda^{-1}\mathbb{E}\int_\Sigma\theta_\varepsilon^{-2}\varphi^{-1}y_\varepsilon^2\,d\sigma dt\\
&+\lambda^3\mathbb{E}\int_{Q_0}\theta^2\varphi^3r_\varepsilon^2\,dxdt+\lambda^2\mathbb{E}\int_{Q}\theta^2\varphi^2R_\varepsilon^2\,dxdt\\
&\leq C\rho\Bigg[\lambda^3\mathbb{E}\int_{Q_0}\theta^2\varphi^3r_\varepsilon^2\,dxdt+\mathbb{E}\int_Q \theta_\varepsilon^{-2}y_\varepsilon^2 \,dxdt\\
&\qquad\quad+\lambda^{-1}\mathbb{E}\int_\Sigma \theta_\varepsilon^{-2}\varphi^{-1}y_\varepsilon^2 \,d\sigma dt+\lambda^2\mathbb{E}\int_{Q}\theta^2\varphi^2R_\varepsilon^2\,dxdt\Bigg]\\
\label{2.0015}&
\quad+C(\rho)\Bigg[\lambda^3\mathbb{E}\int_{Q}\theta^2\varphi^3z^2\,dxdt+\lambda^2\mathbb{E}\int_{\Sigma}\theta^2\varphi^2z^2\,d\sigma dt\Bigg].
\end{aligned}
\end{align}
Taking a small enough $\rho>0$ in \eqref{2.0015} and recalling \eqref{2.003}, we conclude that
\begin{align}
\begin{aligned}
&\,\lambda^{-3}\mathbb{E}\int_{Q}\theta^{-2}\varphi^{-3}u_\varepsilon^2\,dxdt+\lambda^{-2}\mathbb{E}\int_{Q}\theta^{-2}\varphi^{-2}v_\varepsilon^2\,dxdt\\
&+\mathbb{E}\int_{Q} \theta^{-2}_\varepsilon y_\varepsilon^2\,dxdt+\lambda^{-1}\mathbb{E}\int_\Sigma\theta_\varepsilon^{-2}\varphi^{-1}y_\varepsilon^2\,d\sigma dt+\frac{1}{\varepsilon}\mathbb{E}\int_{G}\vert y_\varepsilon(T)\vert^2\,dx\\
\label{2.07}&
\leq C\,\Bigg[\lambda^3\mathbb{E}\int_{Q}\theta^2\varphi^3z^2\,dxdt+\lambda^2\mathbb{E}\int_{\Sigma}\theta^2\varphi^2z^2\,d\sigma dt\Bigg].
\end{aligned}
\end{align}

On the other hand, computing $d(\theta^{-2}_\varepsilon\varphi^{-2}y_\varepsilon^2)$, integrating the obtained equality on $Q$ and taking the expectation on both sides, we get that
\begin{align}\label{2.08}
\begin{aligned}
&\,2 \mathbb{E}\int_Q \theta^{-2}_\varepsilon\varphi^{-2}\mathcal{A}\nabla y_\varepsilon\cdot\nabla y_\varepsilon \,dxdt\\
&=\mathbb{E}\int_Q (\theta^{-2}_\varepsilon\varphi^{-2})_ty_\varepsilon^2\,dxdt-2 \mathbb{E}\int_Q y_\varepsilon \mathcal{A}\nabla y_\varepsilon\cdot\nabla(\theta_\varepsilon^{-2}\varphi^{-2}) \,dxdt\\
&\quad\,+2\lambda^2\mathbb{E}\int_\Sigma \theta_\varepsilon^{-2}\theta^2 y_\varepsilon z \,d\sigma dt+2\mathbb{E}\int_{Q_0}\theta^{-2}_\varepsilon\varphi^{-2}y_\varepsilon u_\varepsilon \,dxdt\\
&\quad\,+2\lambda^3\mathbb{E}\int_Q \theta^{-2}_\varepsilon\theta^2\varphi y_\varepsilon z \,dxdt+\mathbb{E}\int_Q \theta^{-2}_\varepsilon\varphi^{-2}v_\varepsilon^2\,dxdt.
\end{aligned}
\end{align}
It is easy to see that for $\lambda\geq CT^2$
\begin{equation*}
\vert(\theta^{-2}_\varepsilon\varphi^{-2})_t\vert\leq CT\lambda\theta_\varepsilon^{-2},\qquad\quad\vert\nabla(\theta^{-2}_\varepsilon\varphi^{-2})\vert\leq C\lambda\theta^{-2}_\varepsilon\varphi^{-1}.
\end{equation*}
This implies for a large enough $\lambda\geq C(T+T^2)$, we have that 
\begin{align}
\begin{aligned}
&\,\mathbb{E}\int_Q \theta_\varepsilon^{-2}\varphi^{-2}\vert\nabla y_\varepsilon\vert^2\,dxdt\\
&\leq CT\lambda \mathbb{E}\int_Q \theta_\varepsilon^{-2}y_\varepsilon^2\,dxdt+C\mathbb{E}\int_Q \theta^{-2}\varphi^{-2}v_\varepsilon^2\,dxdt\\
&\quad+C\lambda^2\mathbb{E}\int_\Sigma  \theta_\varepsilon^{-1}\theta \vert  z\vert\vert y_\varepsilon\vert \,d\sigma dt+C\lambda\mathbb{E}\int_Q \theta_\varepsilon^{-2}\varphi^{-1}\vert y_\varepsilon\vert\vert\nabla y_\varepsilon\vert \,dxdt\\
\label{2.09}&\quad+C\lambda^3\mathbb{E}\int_Q \theta_\varepsilon^{-1}\theta\varphi\vert z\vert\vert y_\varepsilon\vert \,dxdt+C\mathbb{E}\int_{Q}\theta_\varepsilon^{-1}\theta^{-1}\varphi^{-2}\vert y_\varepsilon\vert\vert u_\varepsilon\vert \,dxdt.
\end{aligned}
\end{align}
Multiplying \eqref{2.09} by $\lambda^{-2}$ and then applying Young's inequality, we obtain that for any $\rho>0$
\begin{align}\label{2.09extra}
\begin{aligned}
&\,\lambda^{-2}\mathbb{E}\int_Q \theta_\varepsilon^{-2}\varphi^{-2}\vert\nabla y_\varepsilon\vert^2\,dxdt\\
&\leq CT\lambda^{-1}\mathbb{E}\int_Q \theta_\varepsilon^{-2}y_\varepsilon^2\,dxdt+C\lambda^{-2}\mathbb{E}\int_Q \theta^{-2}\varphi^{-2}v_\varepsilon^2\,dxdt\\
&\quad+C\lambda^2\mathbb{E}\int_\Sigma  \theta^2\varphi^2  z^2 \,d\sigma dt+C\lambda^{-2}\mathbb{E}\int_\Sigma  \theta_\varepsilon^{-2}\varphi^{-2}  y_\varepsilon^2 \,d\sigma dt\\
&\quad
+\rho\lambda^{-2}\mathbb{E}\int_Q\theta_\varepsilon^{-2}\varphi^{-2}\vert\nabla y_\varepsilon\vert^2 \,dxdt
+C(\rho)\mathbb{E}\int_Q \theta_\varepsilon^{-2} y_\varepsilon^2 \,dxdt\\
&\quad+C\lambda^3\mathbb{E}\int_Q \theta^2\varphi^3 z^2 \,dxdt+C\lambda^{-1}\mathbb{E}\int_Q \theta_\varepsilon^{-2} \varphi^{-1} y_\varepsilon^2 \,dxdt\\
&\quad
+C\lambda^{-4}\mathbb{E}\int_{Q}\theta^{-2}\varphi^{-4} u_\varepsilon^2 \,dxdt+C\mathbb{E}\int_{Q}\theta_\varepsilon^{-2} y_\varepsilon^2 \,dxdt.
\end{aligned}
\end{align}
Hence for a large $\lambda\geq C(T+T^2)$, we get that
\begin{align}\label{2.010}
\begin{aligned}
&\,\lambda^{-2}\mathbb{E}\int_Q \theta_\varepsilon^{-2}\varphi^{-2}\vert\nabla y_\varepsilon\vert^2\,dxdt\\
&\leq C\Bigg[\mathbb{E}\int_Q \theta_\varepsilon^{-2} y_\varepsilon^2 \,dxdt+\lambda^{-1}\mathbb{E}\int_\Sigma \theta_\varepsilon^{-2}\varphi^{-1} y_\varepsilon^2 \,d\sigma dt\\
&\quad\quad\;\;+\lambda^3\mathbb{E}\int_Q \theta^2\varphi^3z^2\,dxdt
+\lambda^2\mathbb{E}\int_\Sigma \theta^2\varphi^2z^2\,d\sigma dt\\
&\quad\quad\;\;+\lambda^{-3}\mathbb{E}\int_{Q} \theta^{-2}\varphi^{-3}u_\varepsilon^2\,dxdt+\lambda^{-2}\mathbb{E}\int_Q \theta^{-2}\varphi^{-2}v_\varepsilon^2\,dxdt\Bigg].
\end{aligned}
\end{align}
Combining \eqref{2.07} and \eqref{2.010}, we deduce that
\begin{align}
\begin{aligned}
&\,\lambda^{-3}\mathbb{E}\int_{Q}\theta^{-2}\varphi^{-3}u_\varepsilon^2\,dxdt+\lambda^{-2}\mathbb{E}\int_{Q}\theta^{-2}\varphi^{-2}v_\varepsilon^2\,dxdt\\
&+\mathbb{E}\int_{Q} \theta^{-2}_\varepsilon y_\varepsilon^2\,dxdt+\lambda^{-1}\mathbb{E}\int_\Sigma \theta_\varepsilon^{-2}\varphi^{-1} y_\varepsilon^2 \,d\sigma dt\\
&+\lambda^{-2}\mathbb{E}\int_Q \theta_\varepsilon^{-2}\varphi^{-2}\vert\nabla y_\varepsilon\vert^2\,dxdt+\frac{1}{\varepsilon}\mathbb{E}\int_{G}\vert y_\varepsilon(T)\vert^2\,dx\\
\label{2.00119}&\leq C\,\Bigg[\lambda^3\mathbb{E}\int_{Q}\theta^2\varphi^3z^2\,dxdt+\lambda^2\mathbb{E}\int_\Sigma \theta^2\varphi^2z^2\,d\sigma dt\Bigg],
\end{aligned}
\end{align}
for $\lambda\geq C(T+T^2)$. Then it follows that there exists  
$$(\widehat{u},\widehat{v},\widehat{y})\in L^2_\mathcal{F}(0,T;L^2(G_0))\times L^2_\mathcal{F}(0,T;L^2(G))\times L^2_\mathcal{F}(0,T;H^1(G)),$$
such that as $\varepsilon\rightarrow0$,
\begin{align}\label{2.020}
\begin{aligned}
&u_\varepsilon \longrightarrow \widehat{u},\,\,\,\textnormal{weakly in}\,\,\,L^2((0,T)\times\Omega;L^2(G));\\
&v_\varepsilon \longrightarrow \widehat{v},\,\,\,\textnormal{weakly in}\,\,\,L^2((0,T)\times\Omega;L^2(G));\\ 
&y_{\varepsilon} \longrightarrow \widehat{y},\,\,\,\textnormal{weakly in}\,\,\,L^2((0,T)\times\Omega;H^1(G)).
\end{aligned}
\end{align}	
Now, let us show that $\widehat{y}$ is the solution of \eqref{2.1} associated with controls $\widehat{u}$ and $\widehat{v}$. To do this, denote by $\widetilde{y}$ the unique solution of \eqref{2.1} with controls $\widehat{u}$ and $\widehat{v}$. For any process $f\in L^2_\mathcal{F}(0,T;L^2(G))$, let $(\phi,\Phi)$ be the solution of the following backward stochastic parabolic equation
\begin{equation*}
\begin{cases}
\begin{array}{lcl}
d\phi+\nabla\cdot(\mathcal{A}\nabla\phi) \,dt = f \,dt+\Phi \,dW(t) &\textnormal{in}& Q,\\
\partial^\mathcal{A}_\nu \phi=0&\textnormal{on}&\Sigma,\\
					\phi(T)=0&\textnormal{in}&G.
				\end{array}
			\end{cases}
		\end{equation*}
By Itô's formula, computing ``$d(\phi y_\varepsilon)-d(\phi\widetilde{y})$", integrating the equality on $Q$, taking the expectation on both sides and letting $\varepsilon\rightarrow0$, we obtain that
\begin{equation*}
\mathbb{E}\int_Q (\widehat{y}-\widetilde{y})f \,dxdt= 0.
\end{equation*}
It follows that $\widehat{y}=\widetilde{y}$ in $Q$, $\mathbb{P}\textnormal{-a.s.}$ Then we deduce that $\widehat{y}$ is the solution of \eqref{2.1} with controls $\widehat{u}$ and $\widehat{v}$. Finally, combining the uniform estimate \eqref{2.00119} and the weak convergence property \eqref{2.020}, we conclude the controllability result of \eqref{2.1} and the desired estimate \eqref{2.2}.
\end{proof}
\section{proof of Proposition \ref{prop3.1}}\label{appendsecB}

We first note that if $F_2\equiv0$, equation \eqref{3.010} becomes a random parabolic equation. Then according to the known Carleman estimate for deterministic parabolic equations in \cite[Theorem 1]{fernandez2006null} and \eqref{2.2012}, we deduce the following Carleman estimate for solutions of \eqref{3.010} with $F\equiv F_2\equiv0$.
\begin{lm}\label{lm3.1}
There exist $\mu_0\geq1$ and a positive constant $C$ depending only on $G$, $G_0$ and $\mathcal{A}$ such that , $F_1\in L^2_\mathcal{F}(0,T;L^2(G))$, $F_3\in L^2_\mathcal{F}(0,T;L^2(\Gamma))$, and $z_0\in L^2(G)$, the corresponding solution of \eqref{3.010} $($with $F\equiv F_2\equiv0$$)$ satisfies that
\begin{align}
\begin{aligned}\label{3.20102}
&\,\lambda^3\mu^4\mathbb{E}\int_Q \theta^2\varphi^3z^2\,dxdt+\lambda\mu^2\mathbb{E}\int_Q \theta^2\varphi \vert\nabla z\vert^2\,dxdt+\lambda^2\mu^3\mathbb{E}\int_\Sigma \theta^2\varphi^2z^2\,d\sigma dt\\
&\leq\, C\Bigg[\lambda^3\mu^4\mathbb{E}\int_{Q_0} \theta^2\varphi^3z^2\,dxdt+\mathbb{E}\int_Q\theta^2F_1^2\,dxdt+\lambda\mu\mathbb{E}\int_\Sigma\theta^2\varphi F_3^2\,d\sigma dt\Bigg],
\end{aligned}
\end{align}
for all $\mu\geq\mu_0$ and $\lambda\geq C(e^{C\mu}T+T^2)$.
\end{lm}
Now, let us proceed with the proof of Proposition \ref{prop3.1}.
\begin{proof}[Proof of Proposition \ref{prop3.1}]
Here, the proof will follow a similar approach to the previous one of Proposition \ref{prop2.4}. Therefore, we will omit some details. Fix $\mu=\mu_0$ as provided in Lemma \ref{lm3.1} and any $\varepsilon>0$, we utilize the same weight functions $\alpha_\varepsilon$ and $\theta_\varepsilon$ defined in \eqref{alptheeps} with the new given $\mu_0$. Now, let us introduce the following optimal control problem
\begin{equation}\label{ch23.96}
\inf\{J_\varepsilon(u)\,\,\vert\,\,u\in\mathcal{U}\},
\end{equation}
where
\begin{align*}\begin{aligned}
			 J_\varepsilon(u)=&\;\frac{1}{2}\mathbb{E}\int_{Q_0} \lambda^{-3}\theta^{-2}\varphi^{-3}u^2 \,dx dt+\frac{1}{2}\mathbb{E}\int_Q \theta_\varepsilon^{-2} y^2 \,dx dt\\
			&+\frac{1}{2}\mathbb{E}\int_\Sigma \lambda^{-1}\theta_\varepsilon^{-2}\varphi^{-1}y^2\,d\sigma dt+\frac{1}{2\varepsilon}\mathbb{E}\int_G |y(0)|^2 \,dx,
		\end{aligned}\end{align*}
and
		$$\mathcal{U}=\bigg\{u\in L^2_\mathcal{F}(0,T;L^2(G_0))\,,\,\,\quad\mathbb{E}\int_{Q_0} \theta^{-2}\varphi^{-3}u^2 \,dx dt<\infty\bigg\}.$$

It is easy to see that \eqref{ch23.96} admits a unique optimal solution $u_\varepsilon$. This optimal solution can be characterized by
\begin{equation}\label{B23.97}
u_\varepsilon = \mathbbm{1}_{G_0}\lambda^3\theta^2\varphi^3r_\varepsilon\,\;\,\textnormal{in}\,\,\,Q\,,\,\,\,\mathbb{P}\textnormal{-a.s.},
\end{equation}
with $r_\varepsilon$ is the solution of the following random parabolic equation
		\begin{equation}\label{B23.98}
			\begin{cases}
				\begin{array}{ll}
					dr_\varepsilon - \nabla\cdot(\mathcal{A}\nabla r_\varepsilon) \,dt = \theta_\varepsilon^{-2}y_\varepsilon \,dt  &\textnormal{in}\,\,Q,\\
					\partial^\mathcal{A}_\nu r_\varepsilon= \lambda^{-1}\theta^{-2}_\varepsilon\varphi^{-1}y_\varepsilon&\textnormal{on}\,\,\Sigma,\\
					r_{\varepsilon}(0)=\frac{1}{\varepsilon}y_\varepsilon(0) &\textnormal{in}\,\,G,
				\end{array}
			\end{cases}
		\end{equation}
where $(y_\varepsilon,Y_{\varepsilon})$ is the solution of \eqref{3.6} with the control $u_\varepsilon$. 

Computing $d(y_\varepsilon r_\varepsilon)$, we obtain that
\begin{align*}
\begin{aligned}
&\;\mathbb{E}\int_Q \theta_\varepsilon^{-2} y_\varepsilon^2 \,dxdt+\lambda^3\mathbb{E}\int_{Q_0}\theta^2\varphi^3 r_\varepsilon^2\,dxdt\\
&+ \lambda^{-1}\mathbb{E}\int_\Sigma \theta^{-2}_\varepsilon\varphi^{-1} y_\varepsilon^2 \,d\sigma dt+ \frac{1}{\varepsilon}\mathbb{E}\int_G |y_\varepsilon(0)|^2\,dx\\
&= -\lambda^3\mathbb{E}\int_Q \theta^2\varphi^3r_\varepsilon z \,dx dt-\lambda^2\mathbb{E}\int_\Sigma \theta^2\varphi^2r_\varepsilon z \,d\sigma dt.
\end{aligned}
\end{align*}
Applying Young's inequality and using Carleman estimate \eqref{3.20102} for equation \eqref{B23.98}, it follows that for any $\rho>0$ 
\begin{align}\label{B701}
\begin{aligned}
&\;\mathbb{E}\int_Q \theta_\varepsilon^{-2} y_\varepsilon^2 \,dxdt+\lambda^3\mathbb{E}\int_{Q_0}\theta^2\varphi^3 r_\varepsilon^2\,dxdt\\
&+ \lambda^{-1}\mathbb{E}\int_\Sigma \theta^{-2}_\varepsilon\varphi^{-1} y_\varepsilon^2 \,d\sigma dt+ \frac{1}{\varepsilon}\mathbb{E}\int_G |y_\varepsilon(0)|^2\,dx\\
&\leq C\rho\Bigg[\lambda^3\mathbb{E}\int_{Q_0} \theta^2\varphi^3r_\varepsilon^2 \,dx dt+\mathbb{E}\int_Q \theta_\varepsilon^{-2} y_\varepsilon^2 \,dxdt+\lambda^{-1}\mathbb{E}\int_\Sigma \theta^{-2}_\varepsilon\varphi^{-1} y_\varepsilon^2 \,d\sigma dt\Bigg]\\
&\quad+C(\rho)\Bigg[\lambda^3\mathbb{E}\int_Q \theta^2\varphi^3 z^2 \,dx dt+\lambda^2\mathbb{E}\int_\Sigma \theta^2\varphi^2 z^2 \,d\sigma dt\Bigg].
\end{aligned}
\end{align}
Taking a small enough $\rho>0$ in \eqref{B701} and using \eqref{B23.97}, we obtain that
\begin{align}\label{B801}
\begin{aligned}
&\;\mathbb{E}\int_Q \theta_\varepsilon^{-2} y_\varepsilon^2 \,dxdt+\lambda^{-3}\mathbb{E}\int_{Q}\theta^{-2}\varphi^{-3} u_\varepsilon^2\,dxdt\\
&+ \lambda^{-1}\mathbb{E}\int_\Sigma \theta^{-2}_\varepsilon\varphi^{-1} y_\varepsilon^2 \,d\sigma dt+ \frac{1}{\varepsilon}\mathbb{E}\int_G |y_\varepsilon(0)|^2\,dx\\
&\leq C\Bigg[\lambda^3\mathbb{E}\int_Q \theta^2\varphi^3 z^2 \,dx dt+\lambda^2\mathbb{E}\int_\Sigma \theta^2\varphi^2 z^2 \,d\sigma dt\Bigg].
\end{aligned}
\end{align}
On the other hand, we calculate $d(\theta_\varepsilon^{-2}\varphi^{-2}y_\varepsilon^2)$, then we get that
 \begin{align}\label{B019}
 \begin{aligned}
   &\;2\mathbb{E}\int_Q \theta_\varepsilon^{-2}\varphi^{-2}\mathcal{A}\nabla y_\varepsilon\cdot\nabla y_\varepsilon  \,dx dt+\mathbb{E}\int_Q \theta_\varepsilon^{-2}\varphi^{-2}Y_\varepsilon^2\,dx dt\\
   & = -\mathbb{E}\int_Q(\theta_\varepsilon^{-2}\varphi^{-2})_t y_\varepsilon^2 \,dx dt
	-2\lambda^3\mathbb{E}\int_Q \theta_\varepsilon^{-2} \theta^2\varphi y_\varepsilon z \,dx dt\\
&\;\;\;-2\mathbb{E}\int_{Q_0}\theta_\varepsilon^{-2}\varphi^{-2}y_\varepsilon u_\varepsilon \,dx dt-2\mathbb{E}\int_Q y_\varepsilon \mathcal{A}\nabla y_\varepsilon\cdot\nabla(\theta_\varepsilon^{-2}\varphi^{-2}) \,dx dt\\
&\;\;\;-2\lambda^2\mathbb{E}\int_\Sigma \theta_\varepsilon^{-2}\theta^2 z y_\varepsilon \,d\sigma dt.
\end{aligned}
\end{align}
Note that for $\lambda\geq CT^2$, we have
\begin{equation}\label{B010}
\vert(\theta^{-2}_\varepsilon\varphi^{-2})_t\vert\leq CT\lambda \theta_\varepsilon^{-2},\qquad\quad\vert\nabla(\theta^{-2}_\varepsilon\varphi^{-2})\vert\leq C\lambda\theta^{-2}_\varepsilon\varphi^{-1}.
\end{equation}
By \eqref{B010} and Young's inequality, \eqref{B019} implies that
 \begin{align}\label{B0110}
 \begin{aligned}
&\;\mathbb{E}\int_Q \theta_\varepsilon^{-2}\varphi^{-2}|\nabla y_\varepsilon|^2  \,dx dt+\mathbb{E}\int_Q \theta_\varepsilon^{-2}\varphi^{-2}Y_\varepsilon^2\,dx dt\\
& \leq CT\lambda \mathbb{E}\int_Q \theta_\varepsilon^{-2} y_\varepsilon^2 \,dx dt+C\lambda^3\mathbb{E}\int_Q \theta_\varepsilon^{-1} \theta \varphi |y_\varepsilon| |z| \,dx dt\\
&\;\;\;+C\mathbb{E}\int_{Q}\theta_\varepsilon^{-1}\theta^{-1}\varphi^{-2} |y_\varepsilon| |u_\varepsilon| \,dx dt+C\lambda\mathbb{E}\int_Q \theta_\varepsilon^{-2}\varphi^{-1} |y_\varepsilon| |\nabla y_\varepsilon| \,dx dt\\
&\;\;\;+C\lambda^2\mathbb{E}\int_\Sigma \theta_\varepsilon^{-1}\theta |z| |y_\varepsilon| \,d\sigma dt.
\end{aligned}
\end{align}
Multiplying \eqref{B0110} by $\lambda^{-2}$ and applying Young's inequality, we get that for any $\rho>0$
 \begin{align}\label{B12012}
 \begin{aligned}
&\;\lambda^{-2}\mathbb{E}\int_Q \theta_\varepsilon^{-2}\varphi^{-2}|\nabla y_\varepsilon|^2  \,dx dt+\lambda^{-2}\mathbb{E}\int_Q \theta_\varepsilon^{-2}\varphi^{-2}Y_\varepsilon^2\,dx dt\\
& \leq CT\lambda^{-1} \mathbb{E}\int_Q \theta_\varepsilon^{-2} y_\varepsilon^2 \,dx dt+C\lambda^3\mathbb{E}\int_Q \theta^2 \varphi^3 z^2 \,dx dt+C\lambda^{-1}\mathbb{E}\int_Q \theta_\varepsilon^{-2} \varphi^{-1} y_\varepsilon^2 \,dx dt\\
&\;\;\;+C\lambda^{-4}\mathbb{E}\int_{Q}\theta^{-2}\varphi^{-4} u_\varepsilon^2 \,dx dt+C\mathbb{E}\int_{Q}\theta_\varepsilon^{-2} y_\varepsilon^2 \,dx dt+\rho\lambda^{-2}\mathbb{E}\int_Q \theta_\varepsilon^{-2}\varphi^{-2} |\nabla y_\varepsilon|^2 \,dx dt\\
&\;\;\;+C(\rho)\mathbb{E}\int_Q \theta_\varepsilon^{-2} y_\varepsilon^2 \,dx dt+C\lambda^2\mathbb{E}\int_\Sigma \theta^2\varphi^2 z^2 \,d\sigma dt+C\lambda^{-2}\mathbb{E}\int_\Sigma \theta_\varepsilon^{-2}\varphi^{-2} y_\varepsilon^2 \,d\sigma dt.
\end{aligned}
\end{align}
Choosing a small enough $\rho>0$ in \eqref{B12012} and taking a large $\lambda\geq C(T+T^2)$, we find that
 \begin{align}\label{B1230}
 \begin{aligned}
&\;\lambda^{-2}\mathbb{E}\int_Q \theta_\varepsilon^{-2}\varphi^{-2}|\nabla y_\varepsilon|^2  \,dx dt+\lambda^{-2}\mathbb{E}\int_Q \theta_\varepsilon^{-2}\varphi^{-2}Y_\varepsilon^2\,dx dt\\
& \leq C\Bigg[\mathbb{E}\int_Q \theta_\varepsilon^{-2} y_\varepsilon^2 \,dx dt+\lambda^3\mathbb{E}\int_Q \theta^2 \varphi^3 z^2 \,dx dt+\lambda^{-3}\mathbb{E}\int_{Q}\theta^{-2}\varphi^{-3} u_\varepsilon^2 \,dx dt\\
&\qquad\;\,+\lambda^2\mathbb{E}\int_\Sigma \theta^2\varphi^2 z^2 \,d\sigma dt+\lambda^{-1}\mathbb{E}\int_\Sigma \theta_\varepsilon^{-2}\varphi^{-1} y_\varepsilon^2 \,d\sigma dt\Bigg].
\end{aligned}
\end{align}
Combining \eqref{B1230} and \eqref{B801}, we conclude that
\begin{align}\label{B140}
\begin{aligned}
&\;\mathbb{E}\int_Q \theta_\varepsilon^{-2} y_\varepsilon^2 \,dxdt+\lambda^{-3}\mathbb{E}\int_{Q}\theta^{-2}\varphi^{-3} u_\varepsilon^2\,dxdt\\
&+ \lambda^{-1}\mathbb{E}\int_\Sigma \theta^{-2}_\varepsilon\varphi^{-1} y_\varepsilon^2 \,d\sigma dt+\lambda^{-2}\mathbb{E}\int_Q \theta_\varepsilon^{-2}\varphi^{-2}|\nabla y_\varepsilon|^2  \,dx dt\\
&+\lambda^{-2}\mathbb{E}\int_Q \theta_\varepsilon^{-2}\varphi^{-2}Y_\varepsilon^2\,dx dt+ \frac{1}{\varepsilon}\mathbb{E}\int_G |y_\varepsilon(0)|^2\,dx\\
&\leq C\Bigg[\lambda^3\mathbb{E}\int_Q \theta^2\varphi^3 z^2 \,dx dt+\lambda^2\mathbb{E}\int_\Sigma \theta^2\varphi^2 z^2 \,d\sigma dt\Bigg].
\end{aligned}
\end{align}
From \eqref{B140}, we deduce that there exists  
$$(\widehat{u},\widehat{y},\widehat{Y})\in L^2_\mathcal{F}(0,T;L^2(G_0))\times L^2_\mathcal{F}(0,T;H^1(G))\times L^2_\mathcal{F}(0,T;L^2(G)),$$
such that as $\varepsilon\rightarrow0$,
\begin{align}\label{B150}
\begin{aligned}
&u_\varepsilon \longrightarrow \widehat{u},\,\,\,\textnormal{weakly in}\,\,\,L^2((0,T)\times\Omega;L^2(G));\\
&y_\varepsilon \longrightarrow \widehat{y},\,\,\,\textnormal{weakly in}\,\,\,L^2((0,T)\times\Omega;H^1(G));\\ 
&Y_{\varepsilon} \longrightarrow \widehat{Y},\,\,\,\textnormal{weakly in}\,\,\,L^2((0,T)\times\Omega;L^2(G)).
\end{aligned}
\end{align}
We now prove that $(\widehat{y},\widehat{Y})$ is the solution of equation \eqref{3.6} associated with the control $\widehat{u}$. For that, let us denote by $(\widetilde{y},\widetilde{Y})$ the unique solution of \eqref{3.6} with the control $\widehat{u}$. For any processes $f,g\in L^2_\mathcal{F}(0,T;L^2(G))$, let $\phi$ be the solution of the following forward stochastic parabolic equation
\begin{equation*}
\begin{cases}
\begin{array}{lcl}
d\phi-\nabla\cdot(\mathcal{A}\nabla\phi) \,dt = f \,dt+g \,dW(t) &\textnormal{in}& Q,\\
\partial^\mathcal{A}_\nu \phi=0&\textnormal{on}&\Sigma,\\
					\phi(0)=0&\textnormal{in}&G.
				\end{array}
			\end{cases}
		\end{equation*}
Calculating ``$d(\phi y_\varepsilon)-d(\phi\widetilde{y})$", integrating the equality on $Q$, taking the expectation on both sides and letting $\varepsilon\rightarrow0$, we end up with
\begin{equation*}
\mathbb{E}\int_Q (\widehat{y}-\widetilde{y})f \,dxdt+\mathbb{E}\int_Q (\widehat{Y}-\widetilde{Y})g \,dxdt= 0.
\end{equation*}
Then it follows that $\widehat{y}=\widetilde{y}$ in $Q$, $\mathbb{P}\textnormal{-a.s.}$, and $\widehat{Y}=\widetilde{Y}$ in $Q$, $\mathbb{P}\textnormal{-a.s.}$ Thus, $(\widehat{y},\widehat{Y})$ is the solution of \eqref{3.6} with the control $\widehat{u}$. Finally, by combining the uniform estimate \eqref{B140} and the weak convergence result \eqref{B150}, we conclude the proof of Proposition \ref{prop3.1}.
\end{proof}

\section*{Acknowledgments}
The authors would like to thank Professor Qi L{\"u} for fruitful discussions and helpful comments.

\end{document}